\documentclass[11pt]{article}
\usepackage{graphicx, tikz}

\usepackage{amssymb, amsmath}
\usepackage{theorem}
\usepackage[colorlinks]{hyperref}

\numberwithin{equation}{section}

\newtheorem{thm}{Theorem}[section]
\newtheorem{lemma}[thm]{Lemma}

\newtheorem{prop}[thm]{Proposition}
\newtheorem{cor}[thm]{Corollary}
{\theorembodyfont{\rmfamily}

\newtheorem{rmk}[thm]{Remark}
}

\newcommand{\qed}{\hfill \mbox{\raggedright \rule{.07in}{.1in}}}

\newenvironment{proof}{\vspace{1ex}\noindent{\bf
Proof}\hspace{0.5em}}{\hfill\qed\vspace{1ex}}
\newenvironment{pfof}[1]{\vspace{1ex}\noindent{\bf Proof of
#1}\hspace{0.5em}}{\hfill\qed\vspace{1ex}}

\newcommand{\R}{{\mathbb R}}
\newcommand{\C}{{\mathbb C}}

\newcommand{\N}{{\mathbb N}}
\renewcommand{\P}{{\mathbb P}}

\newcommand{\B}{{\mathcal B}}
\newcommand{\M}{{\mathcal M}}
\newcommand{\ii}{\mathfrak{i}}
\newcommand{\dd}{\mathrm{d}}
\newcommand{\E}{\mathbb{E}}

\def\eps{\varepsilon}

\def\ae{{\em a.e.} }

\title{Darling--Kac theorem for renewal shifts \\ in the absence of 
regular variation}

\author{P\'eter Kevei\thanks{MTA-SZTE Analysis and Stochastics Research Group,
Bolyai Institute, Aradi v\'ertan\'uk tere 1, 6720 Szeged, Hungary; email: 
kevei@math.u-szeged.hu}\quad and\quad Dalia Terhesiu\thanks{Department of Mathematics, 
University of Exeter, North Park Road 
Exeter, UK, EX4 4QF; email: daliaterhesiu@gmail.com }}

\begin{document}

\maketitle

\begin{abstract}
We study null recurrent renewal Markov chains with renewal distribution in the domain of 
geometric partial attraction of a 
semistable law. Using the classical procedure of inversion, we derive a limit theorem 
similar to the Darling--Kac law along subsequences and obtain some interesting 
properties of  the limit distribution. Also in this context, we obtain a Karamata type 
theorem along subsequences for positive operators. In both results, we identify the 
allowed class of subsequences. We provide several examples of nontrivial  infinite 
measure preserving systems to which these results apply.
\end{abstract}

\section{Introduction and summary of main results}\label{sec-intro}

We recall that regular variation is an essential condition for the existence of a
Darling--Kac law~\cite{DarlingKac57}.  Restricting to  the simple setting of one-sided 
null recurrent renewal chains, our aim is to understand what happens if the regular 
variation is replaced by a weaker assumption on the involved  `renewal' distribution. As 
we explain in the sequel, we will assume that this distribution is in the domain of  
geometric partial attraction of a semistable law, a subclass of infinitely divisible 
laws. Among the main references for ground results on semistable laws, we recall that 
the behaviour of the associated characteristic function has been first understood 
by Kruglov \cite{Kruglov} and that a probabilistic approach in understanding  such laws 
has been developed by Cs\"org\H{o} \cite{Csorgo1}. For more recent advances on `merging 
results' we refer to Cs\"org\H{o} and Megyesi \cite{CM1}, Kevei \cite{PK1}, and 
references therein.

The classical Darling--Kac law for one-sided null recurrent renewal shifts / Markov 
chains is 
recalled in Subsection~\ref{sec-DKst}. The analogue of this law in the semistable setting 
is contained in Section~\ref{sec-dualityss}; this is the content of 
Theorem~\ref{thm-merge}. Several properties of the limit distribution appearing in 
Theorem~\ref{thm-merge} are discussed in Section~\ref{sec-distrib}. In particular, we 
study the asymptotic behaviour of this distribution at $0$ and $\infty$. Although, as 
recalled in Section~\ref{sec-distrib}, the asymptotic behaviour at $\infty$
can be read off from previous results, we  note the somewhat surprising result 
Theorem~\ref{thm:Hat0} that gives the asymptotic behaviour
of this distribution at $0$.  
In Section~\ref{sec-renmeas} we determine the asymptotics of the renewal function in the 
semistable setup, and extend this result for positive operators.
In Section~\ref{sec-examples} we provide
a number of examples (notably, perturbed Wang maps and piecewise linear Fibonacci maps) to 
which Theorem~\ref{thm-merge} applies. The examples considered in 
Section~\ref{sec-examples} are dynamical systems that are isomorphic to Markov chains. In 
Section~\ref{sec-dets}, we discuss the application of Theorem~\ref{thm-merge} to specific 
dynamical systems that are not isomorphic to a Markov chain. Finally, some technical 
proofs are contained in the Appendix.

\subsection{Darling--Kac law for null recurrent renewal chains under regular 
variation}
\label{sec-DKst}

Fix a probability distribution $(f_k)_{k \geq 0}$, $\sum_{k=0}^\infty f_k =1$, and
consider the Markov renewal chain  $(X_n)_{n \geq 0}$, $X_n \in \N_0=\N\cup \{0\}$
with transition probabilities
\begin{equation}
\label{eq-plk}
p_{\ell, k} := \P(X_{n+1} = k | X_n = \ell) = 
\begin{cases}
f_k, & \ell = 0,\\
1, & k = \ell-1, \ \ell \geq 1,\\
0, & \text{otherwise.}
\end{cases}
\end{equation}
Clearly, $X_n$ is a recurrent Markov chain, with unique invariant measure
\begin{equation} \label{eq:pi-def}
\pi_n = \pi_0 \sum_{i=n}^{\infty} f_i, \quad n \geq 1, 
\quad \text{and } \, \pi_0 > 0.
\end{equation}
The chain is null recurrent (i.e.~the invariant measure is infinite) if and only if
$\sum_{k=1}^\infty k  f_k = \infty$, which we assume in the following.

Assume that the chain starts from 0, i.e.~$X_0=0$, and let 
$0 = Z_0 < Z_1 < Z_2 < \ldots$ denote the consecutive return times to 0. Since the Markov 
chain is recurrent, all these random variables are a.s.~finite, and by the Markov property
\begin{equation*} \label{eq:return}
Z_n = \tau_1 + \tau_2 + \ldots + \tau_n, \quad n \geq 1,
\end{equation*}
where $\tau, \tau_1, \tau_2, \ldots$ are iid random variables, with
distribution $\P(\tau = k ) = f_{k-1}$, $k\geq 1$.

Let
\begin{equation*} \label{eq:S-def}
S_n = \sum_{j=0}^{n-1} 1_{X_j=0}, \quad n \geq 1,
\end{equation*}
denote the occupation time of $0$, i.e.~the number of visits to $0$ up to time $n-1$.
Recall the duality rule between $S_n$ and $Z_m$
\begin{equation} \label{eq-duality}
S_n \geq m \ \, \Longleftrightarrow \ Z_{m-1} \leq n-1,
\end{equation}
which means that the number of visits to the state $0$ before time 
$n$ is at least $m$ if and only if the  $(m-1)$st return takes place before time $n$.

\medskip

Up to now everything holds true for a general recurrent Markov renewal chain. In what 
follows, we recall how a distributional limit theorem for $Z_n$ translates to a limit 
theorem for $S_n$. To do so, we assume that $(f_j )_{j \geq 0}$ is  in the domain of 
attraction of an $\alpha$-stable, $\alpha <1$; that is,
\begin{equation*} \label{eq:stable-ass}
\sum_{j\geq n} f_j=\ell(n)n^{-\alpha},
\end{equation*}
for a slowly varying function $\ell$. Then
\begin{align} \label{eq-stable}
\frac{Z_n}{n^{1/\alpha}\ell_1(n)}\rightarrow^d Z_\alpha,
\end{align}
with the norming sequence $n^{1/\alpha} \ell_1(n)$ being the asymptotic inverse of 
$n^{\alpha}/\ell(n)$, where $Z_\alpha$ is an $\alpha$-stable law and $\to^d$ stands for 
convergence in distribution. 
In the following all nonspecified limit relations are meant as $n \to \infty$.
It is known (see, for instance, Bingham \cite{Bingham1}) that the stable limit law for 
$Z_n$ can be translated into a Darling--Kac law for $S_n$.

Let $\mathcal{M}_\alpha$ be a positive random variable distributed according to the 
normalised Mittag-Leffler distribution of
order $\alpha$, that is  $\E (e^{z\mathcal{M}_\alpha})=\sum_{p=0}^\infty 
\Gamma(1+\alpha)^pz^p/\Gamma(1+p\alpha)$ 
for all $z\in\C$.  We recall that $\M_\alpha=^d (Z_\alpha)^{-\alpha}$ and sketch 
the argument for obtaining a Darling--Kac law from~\eqref{eq-duality} 
and~\eqref{eq-stable}.

Let $b(n)=n^{1/\alpha} \ell_1(n)$ and let 
$a(n) = n^\alpha / \ell(n)$ be its asymptotic inverse, that is 
$a(b(n)) \sim n$. In what follows, to ease notation we suppress the integer part. 
Using~\eqref{eq-duality}, and $b(a(n) x) \sim x^{1/\alpha} n$, $n \to \infty$,
we obtain
\[
\begin{split}
\P (S_n \ge a(n) x)
&= \P (Z_{a(n) x -1 }\le n -1) \\
&= \P \Big(\frac{Z_{ a(n) x -1 }}{b( a(n) x -1 )}\leq  
\frac{n -1 }{b( a(n) x -1) }\Big) \\
& \rightarrow \P ( Z_\alpha \leq x^{-1/\alpha} ) \\
& = \P( \M_\alpha \ge x).
\end{split}
\]
Hence, $S_n/ a(n) \to^d \M_\alpha$,
which gives the Darling--Kac law in this simplified setting.

As already mentioned, in what follows we employ the inversion procedure 
described above weakening the assumption on $\tau$.  Namely, we will assume
that $\tau$ is in the domain of geometric partial attraction of a semistable law of  
order $\alpha\in (0,1)$, as recalled in Section~\ref{sec-semist}.

\subsection{Renewal chain, induced renewal chain}
\label{subsec-renrec}

Put $X = \N_0^{\N_0}$ and let $T:X \to X$ be the shift map.
Introduce the cylinders 
\[
[e_0e_1 \ldots e_{k-1}] := \{ x=(x_0, x_1, \ldots) \in X\, : \, 
x_i = e_i, \ i=0,1, \ldots, k-1 \}.
\]
We define the $T$-invariant measure $\mu$ as
\[
\mu([e_0e_1 \ldots e_{k-1}]) = \mu([e_0]) p_{e_0e_1} \cdots p_{e_{k-2} e_{k-1}},
\]
where $\mu([j]) = \pi_j$ given in \eqref{eq:pi-def}. 
The measure extends uniquely to the $\sigma$-algebra generated by 
the cylinder sets.
For simplicity, we assume that $\mu([0]) = \pi_0 = 1$. 

Let $Y = [0] = \{ x \in X : x_0 = 0\}$, and decompose
$$
Y = \cup_{k \geq 0} C_k, \quad \text{ where } C_k = [0,k,k-1,k-2,\dots, 0].
$$
The cylinders $C_k$ are pairwise disjoint, and their
measures are given by
\[
\mu(C_{k}) = \mu(Y) p_{0,k} p_{k,k-1} \cdots p_{1,0} = f_k.
\] 

We recall the definition of the \emph{induced shift on $Y$} and associated `induced 
renewal chain'. For $y\in Y$, let $\tau(y) = \min\{ n \geq 1 : T^n(y) \in Y\}$ and $T_Y = 
T^\tau$. The probability measure $\nu=\mu(Y)^{-1}\mu|_Y = \mu|_Y$ is  $T_Y$-invariant.
To see this it is enough to show that $\nu(C_k) = \nu(T_Y^{-1} C_k)$ for any $k \geq 0$. 
Noting that
\[
T_Y^{-1}(C_k) = \cup_{\ell=0}^\infty [0,\ell, \ell-1, \ldots, 1, 0, k, k-1, \ldots , 1, 0]
\]
we have
\[
\begin{split}
\nu( T_Y^{-1} C_k )&  = \sum_{\ell=0}^\infty \nu([0,\ell, \ell-1, \ldots, 1, 0, k, k-1, 
\ldots , 1, 0] ) \\
& = \sum_{\ell=0}^\infty \frac{\mu([0]) f_\ell f_k}{\mu([0])} \\
& = f_k = \nu(C_k).
\end{split}
\]

We note that $C_k = \{ y \in Y : \tau(y) = k+1\}$ and that $T_Y$ can be regarded as the 
shift on the space $(\{ C_k \}_{k\ge 0})^{\N_0}$.  Given that $\B_Y$ is the 
$\sigma$-algebra generated by cylinders, the induced  shift $(Y,\B_Y, T_Y,\nu)$ is a 
probability measure preserving transformation.

\subsection{Renewal sequences  and transfer operators associated with  Markov shifts }
\label{ssec-RT}

In the set-up of Subsection~\ref{subsec-renrec}, we recall that  the renewal sequence 
$\{u_n\}_{n\ge 1}$ associated with the recurrent
shift $(X,\B_X, T,\mu)$ is given by
$$
u_0=1,\quad u_n=\sum_{j=1}^n f_j u_{n-j}.
$$
We let   $L: L^1(\mu)\to L^1(\mu)$ be 
the transfer operator associated with the shift $(X,\B_X, T,\mu)$ defined by 
$\int_{ X} L^n v \cdot  w \, \dd \mu=\int_{ X} v \cdot  w\!\circ\!T^n \, \dd \mu$, 
$n\ge 1$, $v \in L^1(\mu)$, $w \in L^\infty(\mu)$.
Roughly, the operator $L$ describes the evolution of (probability) densities under 
the action of $T$. Alternatively, the operator $L$
acting on piecewise constant functions (that is, constant  functions on cylinder sets)  
can be identified with the stochastic matrix  with entries
$p_{\ell, k}$ given in~\eqref{eq-plk}. Moreover, the following holds \ae on $Y=[0]$ 
(for a precise reference, see, for instance, Aaronson~\cite[Proposition 
5.1.2 and p.~157]{Aaronson}),
\begin{equation}
\label{eq-tranren}
p_{0,0}^{(n)}=L^n1_{[0]}=L^n(1_Y)= u_n=\mu(Y\cap T^{-n}Y),
\end{equation}
and the equality $L^n1_{[s]}=\mu([s]\cap T^{-n}[s])$ hold \ae on any cylinder $[s]$.
Under the assumption that the tail sequence $\mu(\tau>n)$ is regularly varying with some 
index in $[0,1]$,
the asymptotic behaviour of the partial sum $\sum_{j=0}^{n-1} u_j=\sum_{j=0}^{n-1} L^j 
1_Y$, 
is well understood; for results in terms of renewal sequences see, for 
instance,~Bingham et~al.\cite[Section 8.6.2]{BGT}; for results stated in terms of both 
average transfer 
operators and 
renewal sequences we refer to~\cite[Chapter 5]{Aaronson}.

The asymptotic behaviour of the partial sum $\sum_{j=0}^{n-1} L^j 1_Y$ has also been 
understood for several classes of infinite measure preserving systems $(X,\B_X, T,\mu)$  
that are not isomorphic to renewal shifts.
Provided the existence of a suitable reference set $Y\subset X$, one considers the return 
time  $\tau$ to $Y$ and obtains a finite measure preserving system $(Y,\B_Y, 
T^\tau,\mu_Y)$.
In case $\mu_Y(\tau>n)$ is regularly varying with index $\alpha<1$,  under certain 
assumptions on $(Y,\B_Y, T^\tau,\mu_Y)$, it has been shown that 
for  $a_n=C\mu_Y(\tau>n)(1+o(1))$, with $C>0$ (depending on the parameters of the map 
$T$), $a_n^{-1}\sum_{j=0}^{n-1} L^j v$  convergences \emph{uniformly} on suitable 
compact subsets  of $X$ and suitable observable $v$.
For a precise statement we refer to the work of Thaler~\cite{Thaler83}; for more recent 
results see Thaler and Zweim\"uller \cite{ThalerZweimuller06}, Melbourne and 
Terhesiu \cite{MT13}, and references therein.

In the present work we assume 
that $\tau$ is in the domain of geometric partial attraction of a semistable law of  
order $\alpha\in (0,1)$ (as in Section~\ref{sec-semist}).
The task is to obtain a Karamata type theorem along subsequences, identifying  the 
allowed class of subsequences. For renewal shifts 
(and implicitly, infinite measure 
preserving systems  that come equipped with an iid sequence $(\tau \!\circ\! T_Y^n)$), 
this type of result was obtained by  
Kevei~\cite{PK2}
and in Section~\ref{sec-renmeas} we recall this result. The new result in this context  
is Theorem~\ref{prop-genav}, which  gives a Karamata type theorem along subsequences for 
positive operators.  In Section~\ref{sec-dets}, we discuss its application to infinite 
measure preserving specific  systems  not isomorphic to renewal shifts; in particular we 
obtain uniform convergence  of the partial sum of transfer operators along subsequences 
on suitable sets.

\section{Semistable laws}
\label{sec-semist}

The class of semistable laws, introduced by Paul L\'evy in 1937, is an important subclass 
of infinitely divisible laws. For definitions, properties, and history of semistable laws 
we refer to Sato \cite[Chapter 13]{Sato}, Meerschaert and Scheffler \cite{MS}, Megyesi 
\cite{M}, Cs\"org\H{o} and Megyesi \cite{CM1}, and the references therein. Here we 
summarise the main results from 
\cite{M, CM1}, and we  specialise these results to nonnegative semistable 
laws.

\subsection{Definition and some properties}
\label{subsec-ss}
Semistable laws are limits of centred and normed sums of iid random variables along 
subsequences $k_n$ for which 
\begin{equation} \label{eq:H} 
k_{n}<k_{n+1} \ \text{for } n\geq 1 \ \text{and } 
\lim_{n\to\infty}\frac{k_{n+1}}{k_n}=c > 1
\end{equation}
hold. Since $c=1$ corresponds to the stable case (\cite[Theorem 2]{M}), we assume 
that $c > 1$. 
The simplest such a sequence is 
\begin{equation*} \label{eq:H2}
k_n = \lfloor c^n \rfloor, 
\end{equation*}
where $\lfloor \cdot \rfloor$ stands for the (lower) integer part.
In what follows we let $c$ be as defined in \eqref{eq:H}. 

The characteristic function of a nonnegative semistable random variable $V$ has the form
\begin{equation*} 
\E e^{\ii t V} =
\exp \left\{ \ii t a + \int_0^\infty (e^{\ii t x } - 1 ) \mathrm{d} R(x) \right\}, 
\end{equation*}
where $a \geq 0$, and 
$M: (0,\infty) \to (0, \infty)$ is a logarithmically periodic function with period 
$c^{1/\alpha} > 1$, i.e.~$M(c^{1/\alpha} x ) = M(x)$ for all $x > 0$, such that 
$- R(x) := M(x)/x^{\alpha}$ is nonincreasing for $x > 0$, $\alpha \in (0,1)$. 
We further assume that $V$ is nonstable, that is $M$ is not constant.

\subsection{Domain of geometric partial attraction}
\label{subsec-DGPA}

In the following $X, X_1, X_2, \ldots $ are iid random variables with distribution 
function $F(x) = \P ( X \leq x )$. 
We fix a semistable random variable $V=V(R)$ with characteristic 
and distribution function 
\begin{equation} \label{eq:semistable-chf-df}
\E e^{\ii t V} =
\exp \left\{  \int_0^\infty (e^{\ii t x } - 1 ) \mathrm{d} R(x) \right\}, \quad 
G(x) = \P ( V \leq x).
\end{equation}
The random variable $X$ belongs to the domain of 
geometric partial attraction of the semistable law $G$ if there is a subsequence $k_n$ 
for which (\ref{eq:H}) holds, and a norming and  a centring sequence $A_n, B_n$, such 
that
\begin{equation} \label{eq-dgp-conv}
\frac{\sum_{i=1}^{k_n} X_i }{A_{k_n}} - B_{k_n} \rightarrow^d V.
\end{equation}

It turns out that without loss of generality we may assume that 
\begin{equation*} \label{eq:An}
A_n = n^{1/\alpha} \ell_1(n)
\end{equation*}
with some slowly varying function $\ell_1$ (see \cite[Theorem 3]{M}).
In order to characterise the domain of geometric partial attraction we need some further 
definitions. As $k_{n+1} / k_n \to c > 1$, for any 
$x$ large enough there is a unique $k_n$ such that $A_{k_n} \leq x < A_{k_{n+1}}$. Define
\begin{equation*} \label{eq:delta}
\delta(x) = \frac{x}{A_{k_n}}.
\end{equation*}
Note that the definition of $\delta$ does depend on the norming sequence.
Finally, let  
\begin{equation} \label{eq:ellell}
x^{-\alpha} \ell(x) : = \sup \{ t : t^{-1/\alpha} \ell_1(1/t) > x \}.
\end{equation}
Then $x^{1/\alpha} \ell_1(x)$ and $y^\alpha / \ell(y)$ are asymptotic inverses of each 
other, and
\begin{equation} \label{eq:ellell2}
x^{1/\alpha} \ell_1(x) \sim \inf \{ y : x^{-1} \geq y^{-\alpha} \ell(y) \}. 
\end{equation}
Thus $\ell$ and $\ell_1$ asymptotically determines each other.
For properties of asymptotic inverse of regularly varying functions we refer to
\cite[Section 1.7]{BGT}.

By Corollary 3 in \cite{M} (\ref{eq-dgp-conv}) holds on the 
subsequence $k_n$ with norming sequence $A_{k_n}$ if and only if 
\begin{equation} \label{eq-dgp-distf}
\overline F(x) := 1 - F(x) = \frac{\ell(x)}{x^\alpha} [ M(\delta(x)) + h(x) ],
\end{equation}
where $h$ is right-continuous error function such that
$\lim_{n \to \infty} h(A_{k_n} x ) = 0$, whenever $x$ is a continuity point of $M$.
Moreover, if $M$ is continuous, then $\lim_{x \to \infty} h(x) = 0$.
(We note that, contrary to the remark after Corollary 3 in \cite{M}, it is not true that 
for the subsequence $k_n = \lfloor c^n \rfloor$ one can replace $\delta(x)$ by $x$ in 
\eqref{eq-dgp-distf}. This holds when $\ell_1(x) \equiv 1$, but not in general.)

Since $\alpha <1$ there is no need for centring in \eqref{eq-dgp-conv}, and we have
\[
\frac{\sum_{i=1}^{k_n} X_i}{k_n^{1/\alpha} \ell_1(k_n)}
\rightarrow^{d} V,
\]
where $V$ has characteristic function as in \eqref{eq:semistable-chf-df}.

\subsection{Possible limits}
\label{subsec-limits}

We assume that for the distribution function of $X$ \eqref{eq-dgp-distf} holds. It turns 
out that on different subsequences there are different limit distributions.
Now we determine the possible limit distributions along subsequences. 
We say that $u_n$ \textit{converges circularly} to 
$u \in (c^{-1},1]$, $u_n \stackrel{cir}{\to} u$, if $u \in (c^{-1},1)$ and $u_n \to u$ 
in the usual sense, or $u =1$ and $u_n$ has limit points $1$, or $c^{-1}$, or both.
For $x > 0$ (large) we define the position parameter as
\begin{equation} \label{eq:def-gamma}
\gamma_x  = \gamma(x) =\frac{x}{k_n}, \quad \text{where } k_{n-1} <  x \leq k_n.
\end{equation}
Note that by \eqref{eq:H}
\[
 c^{-1} =\liminf_{x \to \infty} \gamma_x < \limsup_{x \to \infty } \gamma_x = 1.
\]

The definitions of the parameter $\gamma_n$ and the circular convergence 
follow the definitions in \cite[p.~774 and 776]{KC}, and are slightly different from 
those in \cite{M}.

From Theorem 1 \cite{CM1} we see that (\ref{eq-dgp-conv}) holds along a subsequence 
$(n_r)_{r=1}^\infty$ (instead of $k_n$) if and only if 
$\gamma_{n_r} \stackrel{cir}{\to} \lambda \in (c^{-1}, 1]$ as $r \to \infty$. In this 
case, by \cite[Theorem 1]{CM1} (or directly from the relation
$- R_\lambda(x) = \lim_{r \to \infty} n_r \overline F(A_{n_r} x)$) the L\'evy function of 
the limit
\begin{equation} \label{eq:def_Rlambda}
R_\lambda(x) = -\frac{M(\lambda^{1/\alpha} x)}{x^\alpha}. 
\end{equation}
Recall the notation in \eqref{eq:semistable-chf-df}. For any $\lambda \in (c^{-1}, 1]$ 
let $V_\lambda$ be a semistable random variable with characteristic 
and distribution function 
\begin{equation} \label{eq:semistable-chf-df-lambda}
\E e^{\ii t V_\lambda} =
\exp \left\{  \int_0^\infty (e^{\ii t x } - 1 ) \mathrm{d} R_\lambda(x) \right\}, \quad 
G_\lambda(x) = \P ( V_\lambda \leq x).
\end{equation}

Thus,
\begin{equation} \label{eq-conv-subseq}
\frac{\sum_{i=1}^{n_r} X_i} {n_r^{1/\alpha} \ell_1(n_r)} \to^d V_{\lambda}
\quad \text{as } r \to \infty,
\end{equation}
whenever $\gamma_{n_r} \stackrel{cir}{\to} \lambda$.

\section{Duality argument in the semistable setting}
\label{sec-dualityss}

Let us fix $\alpha \in (0,1)$, $c > 1$, the semistable law $V$ as in 
\eqref{eq:semistable-chf-df}, and a slowly varying function $\ell_1$.
Recall the definitions of $X_n, Z_n$ and $S_n$ from Subsection~\ref{sec-DKst}. 
Then $\tau, \tau_1, \tau_2, \ldots$ is an iid sequence with distribution 
function $F(x)= \P (\tau \leq x)$. Throughout the 
remainder of this paper, we assume that the tail $\overline F=1-F$ 
satisfies \eqref{eq-dgp-distf} for some $k_n$ for which \eqref{eq:H} holds,
and for the slowly varying function $\ell$ defined through $\ell_1$ 
in \eqref{eq:ellell}.\footnote{In fact, here we could assume that $\overline F$ satisfies 
the discrete version of \eqref{eq-dgp-distf} and extend $\ell$ and $h$ such that 
$\overline F$ satisfies \eqref{eq-dgp-distf}; see Section~\ref{sec-discr}.}
We recall that this assumption is equivalent to
\[
\frac{\sum_{i = 1}^{k_n} \tau_i}{k_n^{1/\alpha} \ell_1 (k_n)} 
\rightarrow^{d} V. 
\]
Moreover, note that (\ref{eq-conv-subseq}) holds whenever
$\gamma(n_r) \stackrel{cir}{\to} \lambda$ as $r \to \infty$.

Let $a_n = n^\alpha/\ell(n)$ be the asymptotic inverse of $n^{1/\alpha} \ell_1(n)$, i.e.
\begin{equation} \label{eq:bn}
a_n^{1/\alpha} \ell_1(a_n) \sim n.
\end{equation}
Clearly, $a_n$ can be chosen to be an integer sequence.
Recall the definition of the positional parameter in \eqref{eq:def-gamma}.

\begin{thm} \label{thm-merge}
If $\gamma(a_{n_r}) \stackrel{cir}{\to} \lambda \in (c^{-1}, 1]$, then for any $x > 0$
\begin{equation} \label{eq-def-H}
\lim_{r \to \infty} \P ( S_{n_r}  / a_{n_r} \leq x ) = 
\P \left( (V_{h_\lambda(x)} )^{-\alpha} \leq x \right)
=: H_\lambda(x),
\end{equation}
where
\begin{equation*} \label{eq:def_h}
h_\lambda(x) = \frac{\lambda x}{c^{\lceil \log_c (\lambda x) \rceil }}. 
\end{equation*}
More generally, the following merging result holds
\begin{equation} \label{eq-merge}
\lim_{n \to \infty} \sup_{x > 0} | \P ( S_n  \geq a_n x) -  
\P ( V_{\gamma(a_n x)} \leq x^{-1/\alpha}) | = 0.
\end{equation}
\end{thm}

In particular, it follows that $H_\lambda$ is a distribution function, which is not 
obvious from its definition. We derive some of its properties in the next sections.

\begin{proof}
Put $\lceil \cdot \rceil$ for the upper integer part.
By the duality (\ref{eq-duality}) and our assumption on $F$
\[
\begin{split}
\P ( S_n \geq a_n x) & =  \P ( Z_{\lceil a_n x \rceil -1} \leq n -1 ) \\
& =  \P \left( \frac{Z_{\lceil a_n x \rceil -1}}{(a_n x)^{1/\alpha} \ell_1(a_n x)}
\leq \frac{n-1}{(a_n x)^{1/\alpha} \ell_1(a_n x)} \right) \\
& \sim \P ( V_{\gamma ( a_n x ) } \leq x^{-1/\alpha}),
\end{split}
\]
where we used \eqref{eq:bn}, the merging theorem (\cite[Theorem 2]{CM1}), and the 
continuity of the distribution function of $V_\lambda$ (in fact they are $C^\infty$).
Note that the asymptotic holds uniformly only for $x$ being in a compact set of 
$(0,\infty)$.
Still the merging 
\eqref{eq-merge} holds uniformly in $x$, since as $x \downarrow 0$ both probabilities go 
to 1, while as $x \to \infty$ both go to 0. Thus we have the merging result 
(\ref{eq-merge}).

To derive the limit theorem \eqref{eq-def-H} we need the following simple lemma, whose 
proof is left to the interested reader.

\begin{lemma} \label{lemma-gamma}
If $\gamma(x_n) \stackrel{cir}{\to} \lambda$, and $x_n \to \infty$ then 
\[
\gamma( x_n y ) \stackrel{cir}{\to} 
\frac{\lambda y}{c^{\lceil \log_c (\lambda y) \rceil }}
=h_\lambda(y)
\]
for any $y > 0$.  
\end{lemma}

From (\ref{eq-merge}) we can deduce the limit theorem. Assume that $\gamma(a_{n_r}) 
\stackrel{cir}{\to} \lambda \in (c^{-1}, 1]$. Then for any $x > 0$
\[
\lim_{r \to \infty} \P ( S_{n_r}  / a_{n_r} \leq x ) = 
\P \left( (V_{h_\lambda(x)} )^{-\alpha} \leq x \right)
\]
which is the statement.
\end{proof}

\section{Distribution function}
\label{sec-distrib}

We notice that the distribution function $H_{\lambda}$ given by \eqref{eq-def-H} depends 
on $\alpha\in(0,1)$, but for ease of notation we suppress this dependency.   
Lemma~\ref{lemma:tail-est} below shows 
that as $x\to \infty$,
the tail $\overline H_{\lambda}(x)$ behaves similarly to the tail of the Mittag-Leffler 
distribution. For a direct comparison, see~\cite[Theorem 8.1.12]{BGT}.
As a consequence, in Corollary~\ref{cor-mom} we obtain that $H_{\lambda}$ is uniquely 
determined by its moments (which gives another analogy with the Mittag-Leffler 
distribution).

The main result of this section is Theorem~\ref{thm:Hat0},
which gives the behaviour of $H_\lambda$ at $0$.

\subsection{Behaviour at infinity}

To understand the asymptotic behaviour of $\overline H_{\lambda}(x)$ as $x\to\infty$, we 
first consider the asymptotic behaviour
of $G_\lambda(x)= \P( V_\lambda\le x)$, as $x\to 0$.
The required estimate is the following statement, which is Theorem 1 by Bingham 
\cite{Bingham2}; see also Theorem 2.3 by Kern and Wedrich \cite{KW}.

\begin{lemma} \label{lemma:ssdf0}
There exist $0 < c_1 \leq c_2 < \infty$ such that for any $\lambda \in [1,c]$
\[
- c_1 \leq \liminf_{x \to 0+} x^{\frac{\alpha}{1-\alpha}} \log G_\lambda (x)
\leq \limsup_{x \to 0+} x^{\frac{\alpha}{1-\alpha}} \log G_\lambda (x)
\leq - c_2
\]
\end{lemma}

\begin{lemma} \label{lemma:tail-est}
For $x$ large enough,
there exist $\kappa_1 > \kappa_2 > 0$ (independent of $x$)  such that 
\[
\exp \left\{ - \kappa_1 x^{\frac{1}{1-\alpha}}  \right\} \leq
\overline H_\lambda(x) = 1 - H_\lambda(x) \leq 
\exp \left\{ - \kappa_2 x^{\frac{1}{1-\alpha}}  \right\}.
\]
\end{lemma}

\begin{proof}
Note that 
\begin{equation} \label{eq:Haux1}
H_{\lambda}(x) =\P \left( V_{\lambda x / c^{ \lceil \log_c (\lambda  x) \rceil}}  
\geq x^{-1/\alpha} \right)=1 - G_{h_\lambda(x)}(x^{-1/\alpha}).
\end{equation}

Clearly, to deal with the presence of $h_\lambda(x)$ in 
$G_{h_\lambda(x)}(x^{-1/\alpha})$,
it is enough to consider the case $\lambda x \in (c^{k-1}, c^{k}]$, for 
$k \geq 0$. As $H_\lambda(x)$ is a distribution function, 
$\overline H_{\lambda}(x)$ is decreasing as a function of $x$,
\begin{equation} \label{eq:tail-est-1}
\begin{split}
\overline H_\lambda(x) & \geq \overline H_\lambda(c^k/ \lambda)
= \P \left( 
V_{1}^{-\alpha} > c^k / \lambda \right) \\
& \geq \P \left( V_{1}^{-\alpha} > c x \right) 
= \P \left( V_{1} \leq (cx)^{-1/\alpha} 
\right) .
\end{split}
\end{equation}
Similarly, we obtain
\begin{equation} \label{eq:tail-est-2}
\overline H_\lambda(x)  \leq 
\P \left( V_{1} \leq (x/c)^{-1/\alpha} \right) .
\end{equation}
Combining \eqref{eq:tail-est-1} and \eqref{eq:tail-est-2} we have
\[
G_1\left( (cx)^{-1/\alpha} \right) \leq 
\overline H_\lambda(x)  \leq 
G_1\left( (x/c)^{-1/\alpha} \right),
\]
which, after substituting back into Lemma \ref{lemma:ssdf0} gives the statement.
\end{proof}

As a consequence of the upper bound for $\overline H_{\lambda}$, we obtain that
the distribution function $H_{\lambda}$ is uniquely determined by its moments.
To see this we verify that Shohat and Tamarkin's criterion \cite[Section 8.0.4]{BGT}
is satisfied.

\begin{cor}
\label{cor-mom} Let $M_k=\int_0^\infty x^k \, \mathrm{d} H_{\lambda}(x)$, $k\geq 0$. Then
$\sum_{k=0}^\infty M_{2k}^{-1/2k}=\infty$.
\end{cor}

\begin{proof}
Using the upper bound in Lemma~\ref{lemma:tail-est}, compute that
\begin{align*}
\int_0^\infty x^k \, \mathrm{d} H_{\lambda}(x)
& =-\int_0^\infty x^k \, \dd (1-H_{\lambda}(x)) \\
& = k\int_0^\infty x^{k-1}(1-H_\lambda(x))\, \dd x\\
&\leq k\int_0^\infty x^{k-1} \exp(-k_2x^{\frac{1}{1-\alpha}})\, \dd x.
\end{align*}
But $\overline U(x):=\exp(-k_2x^{\frac{1}{1-\alpha}})$ is precisely the tail of a 
Mittag-Leffler distribution $U$, which is different from the standard Mittag-Leffler 
distribution only in terms of $k_2$; see~\cite[Theorem 8.1.12]{BGT}.
Write
\begin{align*}
k\int_0^\infty x^{k-1} \overline U(x)\, \dd x=\int_0^\infty x^k \, \dd U(x):=m_k
\end{align*}
and note that $m_k$ is the $k$-th moment of a Mittag-Leffler distribution. Also, it 
follows that $M_{k}<m_k$.  It is known that $\sum_{k=0}^\infty m_{2k}^{-1/2k}=\infty$ 
(see, for instance,~\cite[Section 8.11]{BGT}). Hence, $\sum_{k=0}^\infty 
M_{2k}^{-1/2k} \geq \sum_{k=0}^\infty m_{2k}^{-1/2k}=\infty$ .
\end{proof}

\begin{rmk}Since $M_k<m_k$ and $H_\lambda$ is uniquely determined by its moments,
we obtain the Laplace transform of $H_{\lambda}$ is bounded from above by the Laplace 
transform of a Mittag-Leffler function.
\end{rmk}

\subsection{Behaviour at zero}
\label{subsec-behav0}

Next we turn to the behaviour of $H_\lambda$ at 0. 
Since $G_\lambda$ is oscillating at infinity for any $\lambda \in (c^{-1}, 1]$, and 
$H_\lambda(x) = \overline G_{h_\lambda(x)}(x^{-1/\alpha})$ it is natural 
to expect an oscillatory behaviour around 0. Surprisingly, it turns out that the 
oscillation of the index and of the argument cancel each other, and result a regular 
behaviour.

\begin{thm} \label{thm:Hat0}
If $M$ is continuous, then for any $\lambda \in (c^{-1}, 1]$
\[
H_\lambda'(0) = \lim_{x \downarrow 0} \frac{H_\lambda(x)}{x} =
M \left(  \lambda^{1/\alpha} \right).
\]
\end{thm}

\begin{proof}
Recall the definition of $R_\lambda$ in \eqref{eq:def_Rlambda}.
Theorem 1.3 by Shimura and Watanabe \cite{ShiWat} combined with Theorem 1 by  
Embrechts et al.~\cite{EGV} imply that if $M$ is continuous, then 
$G_\lambda$ is subexponential for any $\lambda \in [c^{-1},1]$. In particular, 
as $x \to \infty$
\[
\overline G_\lambda(x) 
 \sim - R_\lambda(x) 
= \frac{M ( x \lambda^{1/\alpha} )}{x^\alpha}.
\]
By Lemma \ref{lemma:uniform} below this holds uniformly in $\lambda \in [c^{-1},1]$.
Recalling \eqref{eq:Haux1} and using the logarithmic periodicity of $M$, we obtain
\[
\begin{split}
H_\lambda(x) & = \overline G_{h_\lambda(x)}( x^{-1/\alpha}) 
\sim - R_{h_\lambda(x)}(x^{-1/\alpha}) \\
& = x M \left(  x^{-1/\alpha} h_\lambda(x)^{1/\alpha} \right) 
= x M \left(  \lambda^{1/\alpha} \right) \quad \text{as } x \downarrow 0,
\end{split}
\]
as stated.
\end{proof}

Here is the uniformity statement, whose technical proof is given in the 
Appendix~\ref{subsec-lemmaunif}.

\begin{lemma} \label{lemma:uniform}
Whenever $M$ is continuous, the asymptotics
\[
\overline G_\lambda(x) 
\sim \frac{M ( x \lambda^{1/\alpha} )}{x^\alpha} \quad \text{as } x \to \infty
\]
holds uniformly in $\lambda \in [1,c]$.
\end{lemma}

\subsection{Example}

For $\alpha \in (0,1)$, let $X, X_1, X_2, \ldots$ be iid random variables with 
distribution $\P ( X = 2^{n/\alpha} ) = 2^{-n}$, $n=1,2,\ldots$. This is the generalised 
St.~Petersburg distribution with parameter $\alpha$; see Cs\"org\H{o} \cite{Csorgo}. 
Short calculation gives that
\[
\overline F(x) =
\P ( X > x ) = \frac{2^{\{ \alpha \log_2 x \}}}{x^\alpha}, \quad x \geq 2^{1/\alpha},
\]
where $\{ \cdot \}$ stands for the fractional part.
Thus, it satisfies \eqref{eq-dgp-distf} with $c = 2$, $k_n = 2^n$, $\ell \equiv 1$, $h 
\equiv 0$, and $M(x) = 2^{\{ \alpha \log_2 x \} }$. 
In this case the positional parameter $\gamma_n$ in \eqref{eq:def-gamma} simplifies as
$\gamma_n = n/ 2^{\lceil \log_2 n \rceil}$, where $\lceil \cdot \rceil$ stands for the 
upper integer part. Thus
\[
\frac{\sum_{i=1}^{n_r} X_i}{n_r^{1/\alpha}} \rightarrow^d W_\lambda, \quad r \to \infty, 
\]
if (and only if) $\gamma_{n_r} \stackrel{cir}{\to} \lambda \in (1/2, 1]$. The L\'evy 
function of the limit is given by
\[
R_\lambda(x) = - \frac{2^{\{ \alpha \log_2 (\lambda^{1/\alpha} x) \}}}{x^\alpha},
\quad  x > 0.
\]

On Figure \ref{fig:1} we see the distribution function of $W_\lambda$ for different 
values of $\lambda$. The oscillatory behaviour of the tail is clearly visible.
Figure \ref{fig:2} shows the corresponding $H_\lambda$ distribution functions.
The distribution functions are calculated by simulation.

\begin{figure}
\begin{center}
\includegraphics[height=8.3cm]{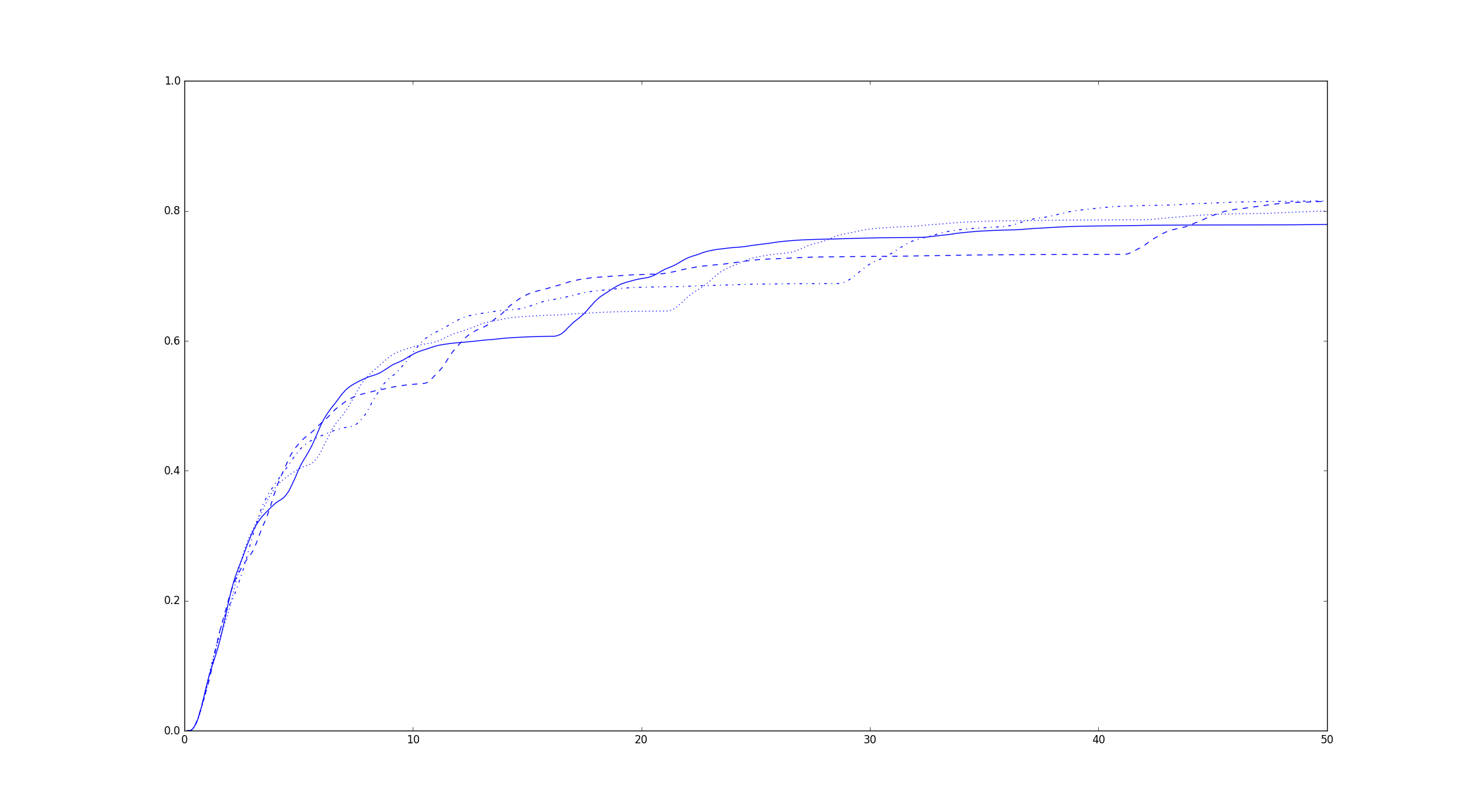}
\vspace*{-20pt}
\caption{The $G_\lambda$ functions in the St.~Petersburg case $\alpha = 0.5$.}
\label{fig:1}
\end{center}
\end{figure}

\begin{figure}
\begin{center}
\includegraphics[height=8.3cm]{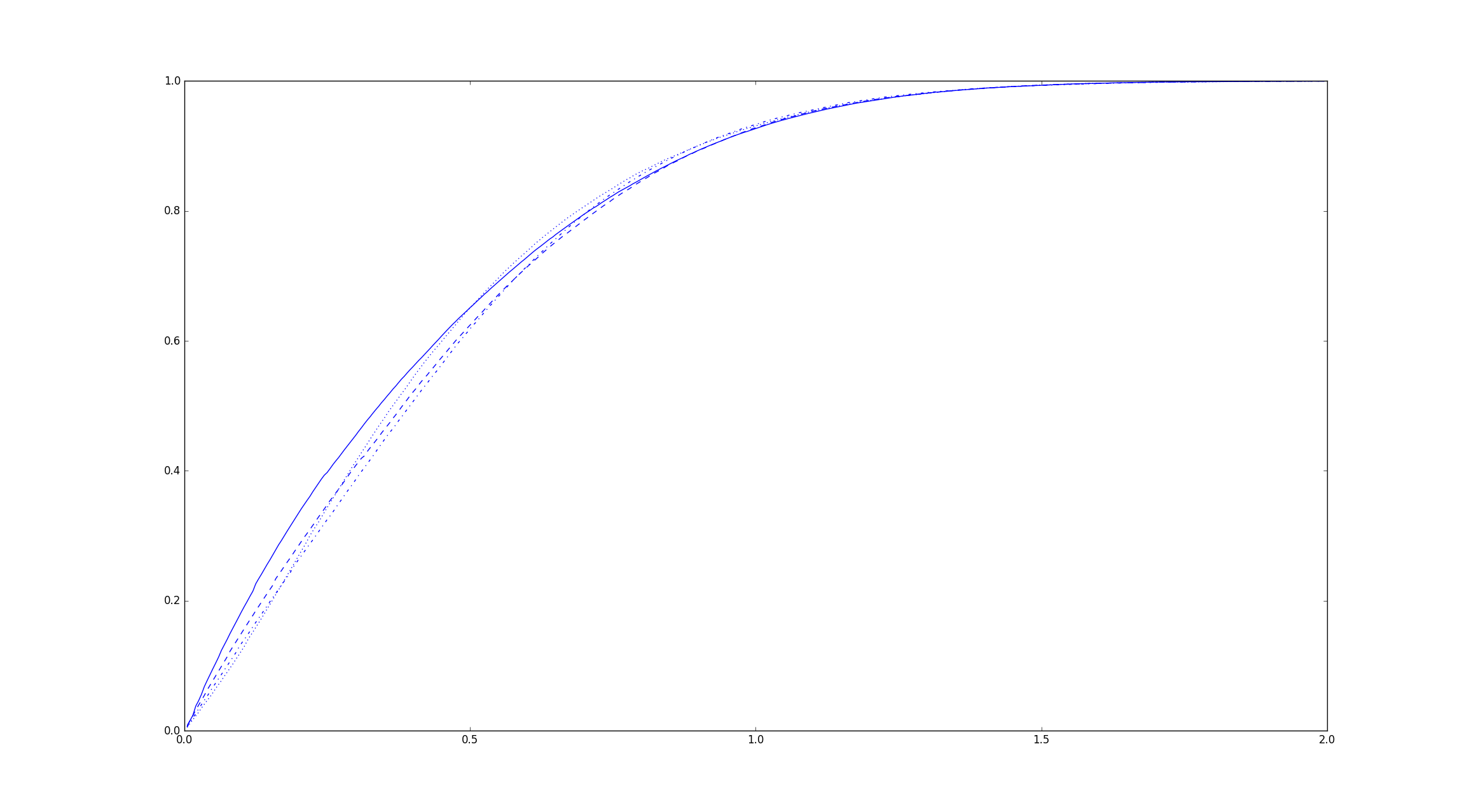}
\vspace*{-20pt}
\caption{The $H_\lambda$ functions in the St.~Petersburg case $\alpha = 0.5$.}
\label{fig:2}
\end{center}
\end{figure}

\section{On the renewal measure}
\label{sec-renmeas}

The aim of this section is to provide asymptotics for the renewal measure of the return 
times when the underlying distribution belongs to the domain of geometric partial 
attraction of a semistable law. We extend the result for positive operators in the spirit 
of Melbourne and Terhesiu \cite{MT12}, which is a crucial step in Section 
\ref{sec-dets} to obtain limit theorems for a dynamical system, which is not isomorphic 
to a Markov renewal chain.

\subsection{Scalar case}

First, we need several definitions and results about regularly 
log-periodic functions; see \cite{PK2}.
Introduce the set of logarithmically periodic functions with period $r > 1$
\[
\begin{split}
\mathcal{P}_{r} = \Big\{  p: (0,\infty) \to (0,\infty)  : & \,
\inf_{x \in [1,r]} p(x) > 0, \ 
p \text{ is bounded, right-continuous, } \\ 
& \text{ and }  p(x r) = p(x), \ 
\forall x >0\Big\}.
\end{split}
\]
Since we need monotonicity, for $r > 1$ we further introduce the sets of functions 
\begin{equation} \label{eq:def-P}
\begin{split}
&\mathcal{P}_{r,\rho} = \Big\{  p: (0,\infty) \to (0,\infty) \, : \,
p \in \mathcal{P}_{r}, \text{ and } x^{\rho} p(x) \text{ is nondecreasing} \Big\}, 
\ \rho \geq 0, \\
&\mathcal{P}_{r,\rho} = \Big\{  p: (0,\infty) \to (0,\infty) \, : \,
p \in \mathcal{P}_{r}, \text{ and } x^{\rho} p(x) \text{ is nonincreasing} \Big\}, 
\ \rho < 0.
\end{split}
\end{equation}
We also need results on the Laplace--Stieltjes transform of regularly log-periodic 
functions. Therefore, for $r > 1$, $\rho \geq 0$, put
\begin{equation*} \label{eq:def-Q}
\begin{split}
\mathcal{Q}_{r,\rho} = \Big\{  q: & (0, \infty) \to (0,\infty) \, :
\, q \in \mathcal{P}_r, \text{ and } s^{-\rho} q(s) \text{ is completely monotone}   
\Big\}.
\end{split}
\end{equation*}
Define the operator $\mathrm{A}_\rho : \mathcal{P}_{r,\rho} \to \mathcal{Q}_{r,\rho}$, 
$\rho > 0$, as
\begin{equation} \label{eq:defA}
\mathrm{A}_\rho p (s) = s^{\rho} \int_0^\infty e^{-s x} \dd ( p(x) x^\rho) .
\end{equation}
In Lemma 1 in \cite{PK2} it is shown that $\mathrm{A}_\rho$ is one-to-one.

Let $\mathcal{P}_{r,\rho}^{1}$ denote the set of differentiable functions in 
$\mathcal{P}_{r,\rho}$. For $r > 1$ and $\rho > 0$ introduce the operator 
$\mathrm{B}_{r, \rho} = \mathrm{B}_{\rho}: \mathcal{P}_{r} \to 
\mathcal{P}_{r,\rho}^{1}$
\begin{equation} \label{eq:defB}
\mathrm{B}_\rho p (x) = x^{-\rho} \int_0^x y^{\rho -1 } p(y) \dd y. 
\end{equation}
Then $\mathrm{B}_\rho$  is  one-to-one with inverse
\begin{equation*} \label{eq:Binv}
{\mathrm{B}}_\rho^{-1} q (x) =  x^{1-\rho} \frac{\dd}{\dd x} [x^\rho q(x)], \quad q \in 
\mathcal{P}_{r,\rho}^{1}.
\end{equation*}

\smallskip

In this section we assume that the subsequence $k_n$ in (\ref{eq:H}) is 
$k_n = \lfloor c^n \rfloor$ and \eqref{eq-dgp-distf} holds with
$\ell \equiv 1$. The latter is equivalent to $\ell_1 \sim 1$ by \eqref{eq:ellell2}.
It is easy to see that in this case $\delta(x)$ can indeed be replaced by
$x$ in \eqref{eq-dgp-distf}. Therefore
\begin{equation} \label{eq:renewal-ass}
\overline F(x) = x^{-\alpha} \left( M(x) + h(x) \right), \quad 
k_n = \lfloor c^n \rfloor,
\end{equation}
where $M(x c^{1/\alpha} ) = M(x)$ for all $x > 0$, i.e.~$M \in 
\mathcal{P}_{c^{1/\alpha}}$, and $\lim_{n \to \infty} h(x c^{n/\alpha}) = 0$ for all $x 
\in C_M$, with $C_M$ being the continuity points of $M$.

The renewal function corresponding to $F$ is defined as
\begin{equation*}
U(x) = \sum_{n=0}^\infty F^{*n}(x),
\end{equation*}
where $F^{*n}$ stands for the $n$th convolution power.

\begin{prop} \label{prop:renewal}
Assume \eqref{eq:renewal-ass}. Then
\begin{equation} \label{eq:U-asy}
\lim_{n \to \infty} 
\frac{U(c^{n/\alpha}z)}{c^n z^\alpha} = p(z), \quad z \in C_p,
\end{equation}
with $p = \mathrm{A}^{-1}_\alpha (
{1}/{\mathrm{A}_{1-\alpha} \mathrm{B}_{1-\alpha} M})$. 
\end{prop}

If $p$ is continuous, then 
(\ref{eq:U-asy})  implies
\[
U(x) \sim x^\alpha  p(x) \quad \text{as } x \to \infty.
\] 

\begin{proof}
From \eqref{eq:renewal-ass} 
\begin{equation*} \label{eq:reglog}
\lim_{n \to \infty} {c^n z^\alpha} \overline F(c^{n/\alpha} z) = M(z),
\quad z \in C_{M}.
\end{equation*}
Corollary 1 in \cite{PK2} implies that
\begin{equation} \label{eq:Laplace-asy}
1 - \widehat F(s) = \int_0^\infty ( 1 - e^{-sy} ) \dd F(y)
\sim s^\alpha q_0(s),
\end{equation}
where $q_0 = \mathrm{A}_{1-\alpha} \mathrm{B}_{1- \alpha} M$.

Thus, using (\ref{eq:Laplace-asy}) for the Laplace transform of $U$ we obtain as 
$s \downarrow 0$
\[
\begin{split}
\widehat U(s) & = \int_0^\infty e^{-sy} \dd U(y) 
= \sum_{n=0}^\infty \left( \widehat F(s) \right)^n \\
& = \frac{1}{1- \widehat F(s)} 
\sim \frac{1}{s^\alpha  q_0(s)}.
\end{split}
\]
By Theorem 1 in \cite{PK2} the latter is equivalent to \eqref{eq:U-asy} with
$\mathrm{A}_\alpha p = 1/q_0$, and the statement follows.
\end{proof}

\subsection{Operator case}
\label{sec-genKaram}

We recall that in the set-up of Subsections~\ref{subsec-renrec} and~\ref{ssec-RT}, we 
have $u_n=L^n 1_Y$, \ae on $Y$, where $L$ is the transfer operator associated with 
$(X,\mathcal{B}_X, T, \mu)$. We assume that \eqref{eq:renewal-ass} holds for the 
distribution function of $\tau$,
 which  by Proposition \ref{prop:renewal}, implies 
\eqref{eq:U-asy}.  As a consequence,
\begin{equation}
\label{eq-avMS}
\lim_{n \to \infty} \frac{\sum_{j=0}^{[ c^{n/\alpha}z]} u_j}{c^n z^\alpha } =
\lim_{n \to \infty} \frac{\sum_{j=0}^{[ c^{n/\alpha}z]} L^n 1_Y}{c^n z^\alpha }
=p(z), \quad z \in C_p.
\end{equation}

In what follows we are interested in a more general form of~\eqref{eq-avMS} that
applies to dynamical system that do not come equipped with an iid sequence 
$\{\tau\circ T_Y^n\}_{n\ge 1}$ and for which \eqref{eq-tranren} does not hold.
We consider such dynamical systems in Section \ref{sec-dets}, where we justify that 
Theorem \ref{prop-genav} below (a generalisation of \eqref{eq-avMS}) applies to them.

Before stating the result of this section, we recall the following notation: we write 
$T(x)\sim c(x)P$  for bounded operators $T(x), P$ acting on some Banach space $\B$ with 
norm $\|\, \|$ if $\|T(x)-c(x) P\| = o(c(x))$.

\begin{thm} \label{prop-genav}
Set $\hat T(e^{-s})=\sum_{n=0}^\infty T_n e^{-sn}$, $s>0$, 
where $T_n$ are uniformly bounded positive operators on some Banach space $\B$ with norm 
$\|\, \|$. Let  $P: \B \to \B$ be a bounded linear operator.  Assume that 
\begin{equation}
\label{eq-hatT}
\hat T(e^{-s})\sim \frac{1}{s^\alpha\ell(1/s) q_0(s)} P \text{ as } s\to 0,
\end{equation}
for some slowly varying function $\ell$, $\alpha\in (0,1)$, and
$q_0 \in \mathcal{Q}_{c^{1/\alpha}, \alpha}$.
Let $p= \mathrm{A}_\alpha^{-1} (1/ q_0)$. 
Then for all  $z \in C_p$, as $n\to\infty$,
$$
\sum_{j=0}^{ \lfloor c^{n/\alpha}z \rfloor} T_j \sim
\frac{c^n z^\alpha}{\ell(c^{n/\alpha} z)} \, p(z)\, P.
$$
\end{thm}

\begin{proof} Given assumption~\eqref{eq-hatT}, we proceed as in the proofs 
of~\cite[Proposition 3.3 and Lemma 3.5]{MT13}, which adapt the proof of Karamata's 
theorem via `approximation by polynomials' (see, for instance, Korevaar \cite[Section 
1.11]{Korevaar}) to the case of positive operators. 

{\bf{Step 1}} Given a polynomial $Q(x)=\sum_{k=1}^m b_k x^k$, we argue that
\begin{equation}
\label{eq-firstQ}
\sum_{j=0}^\infty  T_j Q(e^{-sj})\sim\frac{1}{s^{\alpha}\ell(1/s)}\int_0^\infty 
Q(e^{-x})\, \dd (p(x/s) x^\alpha) P.
\end{equation}
Note that
$$
\hat 
T(e^{-s})\sim\frac{1}{s^{\alpha}\ell(1/s)q_0(s)}P=\frac{1}{s^{\alpha}\ell(1/s)}\mathrm
{A}_\alpha p(s)P=\frac{1}{\ell(1/s)}\int_0^\infty e^{-sx}\, \dd (p(x) x^\alpha)P
$$
and that $\sum_{j=0}^\infty  T_j Q(e^{-sj})=\sum_{k=1}^m b_k \sum_{j=0}^\infty  T_j 
e^{-sj k}=\sum_{k=1}^m b_k \hat T(e^{-sk})$.
Now, for $k\in\N$,
$$
\hat T(e^{-sk})\sim\frac{1}{\ell(1/s)}\int_0^\infty e^{-skx}\, \dd (p(x) 
x^\alpha)P=\frac{1}{s^\alpha\ell(1/s)}\int_0^\infty e^{-kx}\, \dd (p(x/s) x^\alpha)P.
$$
Hence,~\eqref{eq-firstQ} follows from the previous displayed equation after 
multiplication with $b_k$ and  summation over $k$.

{\bf{Step 2}} Let $g=1_{[e^{-1},1]}$. Let $\eps>0$ be arbitrary and let $z$ be a 
continuity point of $p$. Therefore we can choose a $\delta > 0$ such that
\begin{equation} \label{eq:p-cont}
p((1+\delta) z ) (1 + \delta)^\alpha -  p((1-\delta) z -) (1 - \delta)^\alpha < 
\frac{\varepsilon}{2}.
\end{equation}
By Lemma~\ref{lemma-poly} in Appendix~\ref{subsec-lemma1}, for these $\varepsilon$ and 
$\delta$ we can choose a polynomial $Q$ such that $Q \geq g$ on $[0,1]$ and for any 
measure $\mu$ on $(0,\infty)$ such that $\int_0^\infty e^{-x} \mu(\dd x) < \infty$,
\begin{equation} \label{eq-Qgp}
\int_0^\infty \left[ Q(e^{-x}) - g(e^{-x}) \right] \mu( \dd x) \leq 
\varepsilon \int_0^\infty e^{-x} \mu(\dd x) + \mu( ( 1- \delta, 1 + \delta)).
\end{equation}
Using that $Q\geq g$ and~\eqref{eq-firstQ}, we obtain
\begin{align*}
\sum_{j=0}^{ \lfloor c^{n/\alpha}z \rfloor} T_j &
=\sum_{j=0}^\infty T_j g \Big( e^{-\frac{j}{ \lfloor c^{n/\alpha}z \rfloor}} \Big) \\
& \le \sum_{j=0}^\infty T_j Q\Big(e^{-\frac{j}{\lfloor c^{n/\alpha}z \rfloor}}\Big)\\
&\sim \frac{ c^{n}z^{\alpha}}{\ell(c^{n/\alpha}z)}
\int_0^\infty Q(e^{-x})  \dd (p(x \lfloor c^{n/\alpha}z \rfloor ) x^\alpha) \, P.
\end{align*}

We apply \eqref{eq-Qgp} for the measure 
$\mu_n(\dd x) = \dd (p(x \lfloor c^{n/\alpha}z \rfloor ) x^\alpha)$. Since $p$ is bounded
\begin{equation} \label{eq:sup-intp}
\sup_{n \geq 1} \int_0^\infty e^{-x} \dd (p(x \lfloor c^{n/\alpha}z \rfloor ) x^\alpha) 
=: K < \infty.
\end{equation}
Using the monotonicity of $p(x) x^\alpha$, the logarithmic periodicity of $p$, and 
\eqref{eq:p-cont}
\[
\begin{split}
\mu_n( (1-\delta, 1+\delta)) 
& \leq p( (1+\delta) \lfloor c^{n/\alpha}z \rfloor ) (1+\delta)^\alpha - 
p( (1-\delta) \lfloor c^{n/\alpha}z \rfloor ) (1-\delta)^\alpha \\
& = ( \lfloor c^{n/\alpha} z \rfloor )^{-\alpha} \Big[  
p( (1+\delta) \lfloor c^{n/\alpha}z \rfloor ) 
( (1+\delta) \lfloor c^{n/\alpha}z \rfloor )^\alpha \\
& \qquad - p( (1-\delta) \lfloor c^{n/\alpha}z \rfloor ) 
( (1- \delta) \lfloor c^{n/\alpha}z \rfloor )^\alpha \Big] \\
& \leq ( \lfloor c^{n/\alpha} z \rfloor )^{-\alpha}
\Big[  p( (1+\delta) c^{n/\alpha}z ) ( (1+\delta) c^{n/\alpha}z )^\alpha \\
& \qquad - p( (1-\delta) (c^{n/\alpha}z -1) ) ( (1- \delta)(c^{n/\alpha}z -1) )^\alpha 
\Big] \\
& \rightarrow p( (1+\delta) z ) (1+\delta)^\alpha -
p( (1-\delta) z -) (1- \delta)^\alpha < \frac{\varepsilon}{2}.
\end{split}
\]
Thus for $n$ large enough
\begin{equation} \label{eq:p-cont1}
\mu_n( (1-\delta, 1+\delta)) < \varepsilon. 
\end{equation}
Thus, using \eqref{eq-Qgp}, \eqref{eq:sup-intp}, \eqref{eq:p-cont1}, and that $z$ is a 
continuity point of $p$, for $n$ large enough
\[
\begin{split}
\int_0^\infty Q(e^{-x})  \dd (p(x \lfloor c^{n/\alpha}z \rfloor ) x^\alpha)
& \leq \int_0^\infty g(e^{-x}) \dd (p(x \lfloor c^{n/\alpha}z \rfloor ) x^\alpha)
+ \varepsilon (K + 1) \\
& \leq p(z) + \varepsilon (K + 2).
\end{split}
\]
Reverse inequality can be shown similarly. Thus the conclusion follows  since $\eps > 0$ 
is arbitrary.
\end{proof}

\section{Examples of null recurrent renewal
shifts satisfying tail condition (\ref{eq-dgp-distf})}
\label{sec-examples}

In this section we construct three  dynamical systems that can be modelled by null 
recurrent renewal shifts (as described in Section~\ref{sec-intro}) that
satisfy  tail condition (\ref{eq-dgp-distf}).
As such, we  justify that 
Theorem~\ref{thm-merge} (describing the distributional behaviour of $S_{n_r}$) and 
Proposition \ref{prop:renewal} (and thus~\eqref{eq-avMS}, describing the
limit behaviour of the average transfer operator $\sum_{j=0}^{[ c^{n/\alpha}z]} L^j 1_Y$) 
apply to these examples. We recall  that dynamical systems that can be modelled by null 
recurrent renewal shifts have the  property that the sequence 
$\{\tau\circ T_Y^n\}_{n\ge 1}$ is iid. 

The first two examples in Subsections~\ref{subsec-contM} and~\ref{subsec-ncontM} can be 
regarded as perturbations of the intermittent map with linear 
branches preserving an infinite measure, known
as Wang map (Gaspard and Wang \cite{GW88}); an exact form of a (unperturbed) Wang type 
map $T_0:[0,1]\to 
[0,1]$ in terms of the parameter  $\alpha>0$  is given by~\eqref{eq-firstW} with 
$\eps=0$. 
We recall that $T_0$ is a linear version of the smooth 
intermittent map studied by Pomeau and Manneville \cite{PM80} with $T_0'(x)>1$ for all 
$x\in (0,1]$ and 
$T_0'(0)=1$ (so, it is expanding everywhere, but at the so-called indifferent fixed point 
$0$).
When $\alpha<1$, the map $T_0$
preserves an infinite measure, equivalently it is a null recurrent renewal chain, where 
the first return $\tau$ to $Y=[1/2, 1]$ satisfies strict regular variation:
$m(\tau > n) = \frac12 n^{-\alpha}$ for the normalised Lebesgue measure $m$ on
$Y = [\frac12,1]$. We recall that this strict regular variation
implies that $T_0$ satisfies a Darling--Kac law and that 
$n^{-\alpha}\sum_{j=0}^{n} L^j 1_Y\to C$, \ae on $Y$, as $n\to\infty$, for some $C>0$ 
(depending only on the parameters of $T_0$).

As clarified in subsections~\ref{subsec-contM} and~\ref{subsec-ncontM}  a slight 
perturbation of $T_0$ gives rise to different tails $m(\tau > n)$, which are no longer 
regularly varying. Instead, we show that  $m(\tau > n)$ satisfies  tail condition 
(\ref{eq-dgp-distf}) with a continuous and a noncontinuous, respectively,  logarithmic 
periodic function $M$ (identifying the involved sequence $k_n$).
Moreover, while the map in  subsection~\ref{subsec-contM} is differentiable at $0$ from 
the right (so, $0$ is an indifferent fixed point), the map in 
subsection~\ref{subsec-ncontM} is not differentiable at $0$; for this second example we 
justify that we can still speak of `the derivative at $0$ along subsequences' being 
equal to $1$ (see equation~\eqref{eq-secWderiv} and text before it).

In subsection~\ref{subsec-Fib}, we introduce a family of maps $T_\lambda$ 
(as in~\eqref{eq-KFib}) generated out of the  sequence of Fibonacci numbers,
somewhat similar to, but simpler in structure than, the maps studied by Bruin and Todd 
in~\cite{BT12, BT15}.  In short, the maps $T_\lambda$ are Kakutani towers over  linear 
maps (as in~\eqref{eq-Ty})  generated out of the  Fibonacci sequence. As such, they are 
isomorphic to renewal shifts and equation~\eqref{eq-tauFib} says that they are null 
recurrent renewal shifts. As shown  in Proposition~\ref{prop-semfib},  the maps 
$T_\lambda$  satisfy tail condition~\eqref{eq-dgp-distf}, identifying the involved 
sequence $k_n$. This justifies that Theorem \ref{thm-merge} applies to $T_\lambda$.
Moreover, the form of the sequence $k_n$ in   Proposition~\ref{prop-semfib} allows for an 
immediate application of  Proposition \ref{prop:renewal} and thus~\eqref{eq-avMS} (see 
text after the proof of  Proposition~\ref{prop-semfib}).

We believe that Proposition~\ref{prop-semfib}  together with Theorem 8.14 by 
Bruin et al.~\cite{BTT} can 
be used to show that Theorem \ref{thm-merge} applies to the family of countably piecewise 
linear (unimodal) maps with Fibonacci combinatorics  studied in~\cite{BT12, BT15}.
For simplicity of the exposition, in this work we restrict to the self-contained model 
introduced in subsection~\ref{subsec-Fib}.

\subsection{First perturbation of the Wang map: continuous case}
\label{subsec-contM}

Fix $\alpha \in (0,1)$, $c > 1$, $\eps > 0$ and for $n \geq 1$, define
\begin{equation} \label{eq-Wang-xi}
\xi_n = \frac12 n^{-\alpha} \left(1+2\eps 
\sin \left( \frac{2\pi \alpha \log n}{\log c} \right) 
\right).
\end{equation}
Note that $\xi_1 = \frac{1}{2}$. First we show that $\xi_n$ is strictly decreasing. 
Let 
\begin{equation} \label{eq:M-Wang1}
M(x) = \frac{1}{2} \left( 1 + 2\eps 
\sin \left( \frac{2\pi \alpha \log x}{\log c} \right) \right).
\end{equation}
Then $M$ is bounded and bounded away from zero for $\eps < \frac12$, and
$M(c^{1/\alpha}x) = M(x)$ for all $x \in (0,\infty)$.
Furthermore $R(x) := -M(x) x^{-\alpha}$ is continuous and nondecreasing
for small $\eps$. Indeed, short calculation shows that $R'(x) = 2^{-1} x^{-(1+\alpha)} 
(\alpha + O(\eps))$, where $|O(\eps)| \leq 2 \alpha \varepsilon ( 1 + 2 \pi / \log c)$. 
This implies that $\xi_n = -R(n)$ in \eqref{eq-Wang-xi} is decreasing, whenever 
$\varepsilon > 0$ is small enough, which we assume in the following.

Set $\xi_0 = 1$, and define a countably piecewise linear map
\begin{equation}
\label{eq-firstW}
T_\eps(x) = 
\begin{cases} 
\frac{\xi_n-x}{\xi_n-\xi_{n+1}} \xi_n + \frac{x-\xi_{n+1}}{\xi_n-\xi_{n+1}} \xi_{n-1},
& \text{ for } x \in [\xi_{n+1}, \xi_n], n \geq 1,  \\
 2x-1, & \text{ for } x \in (\frac12,1].
\end{cases}
\end{equation}
Then $T_\eps(\xi_n) = \xi_{n-1}$ for $n \geq 1$, and the graph of $T_\eps$ consists
of line segments connecting the points $(\xi_n, \xi_{n-1}) \in [0,1]$ for $n \geq 1$, as 
well as $(\frac12,0)$ to $(1,1)$. For $\eps = 0$ we have exactly the Wang map $T_0$.
The graph of $T_\eps$ for $\eps > 0$ has Hausdorff distance $\leq \eps$ to the graph
of $T_0$ and thus, $\| T_\eps-T_0 \|_\infty \leq \eps$.

Straightforward calculation shows that
\begin{align}
\label{eq-Deltaxi}
\begin{split}
\Delta \xi_n & := \xi_n - \xi_{n+1} \\
& =
\frac{\alpha n^{-\alpha -1}}{2}  \left[ 1 + 2 \varepsilon 
\sin \left( \frac{2 \pi \alpha}{\log c} \log n \right) - 
\frac{4 \pi \varepsilon }{\log c} \cos \left( \frac{2 \pi \alpha}{\log c} \log 
n\right) \right] + O(n^{-2 -\alpha}),
\end{split}
\end{align}
from which we see that $T_\varepsilon$ is differentiable at 0 from the right, and
\[
T_\varepsilon' (0) = 1,
\]
so 0 is an indifferent fixed point.

Let $\tau$ be the first return time to $[1/2,1]$. We see that for $n \geq 1$, 
$\{ \tau = n + 1 \} = ((1+\xi_{n+1})/2, (1+\xi_n)/2]$, thus
$$
m(\tau > n)=\sum_{j\ge n}\Delta \xi_n
= \xi_n = \frac12 n^{-\alpha} 
\left(1+2\eps \sin \left( \frac{2\pi \alpha \log n}{\log c} \right) \right),
$$
where $m$ is the normalised Lebesgue measure on $Y = [\frac12,1]$.

Define $k_n = \lfloor c^n \rfloor$, $\ell \equiv 1$ and $A_k = k^{1/\alpha}$, so 
$A_{k_n} = \lfloor c^{n}\rfloor^{1/\alpha}$. In this case $\delta(x)$ can be simply 
changed to $x$ in \eqref{eq-dgp-distf}. Thus $\tau$ satisfies \eqref{eq-dgp-distf} with 
$M$ in \eqref{eq:M-Wang1}.

Figures \ref{fig:3} and \ref{fig:4} show the limiting $G_\lambda$ and $H_\lambda$ 
functions for the parameter values $\alpha = 0.5$, $\varepsilon = 0.04$, and $c = 2$. The 
distribution function is calculated numerically from the characteristic function using 
the Gil-Pelaez--Ros\'en inversion formula \cite{GilPelaez, Rosen}.

\begin{figure}
\begin{center}
\includegraphics[height=8.3cm]{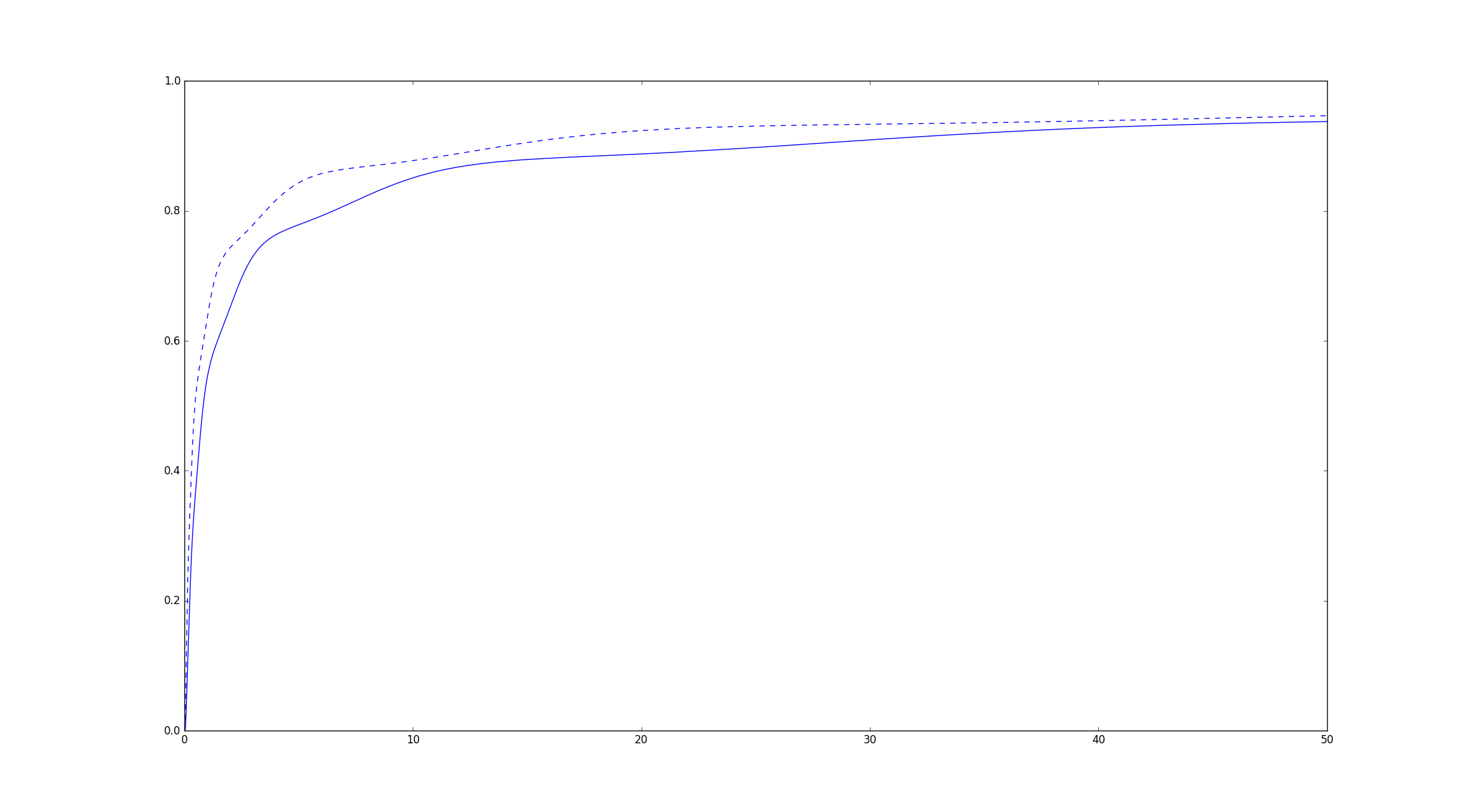}
\vspace*{-20pt}
\caption{The distribution functions $G_{0.5}$ (solid) and $G_{0.75}$ (dashed).}
\label{fig:3}
\end{center}
\end{figure}

\begin{figure}
\begin{center}
\includegraphics[height=8.3cm]{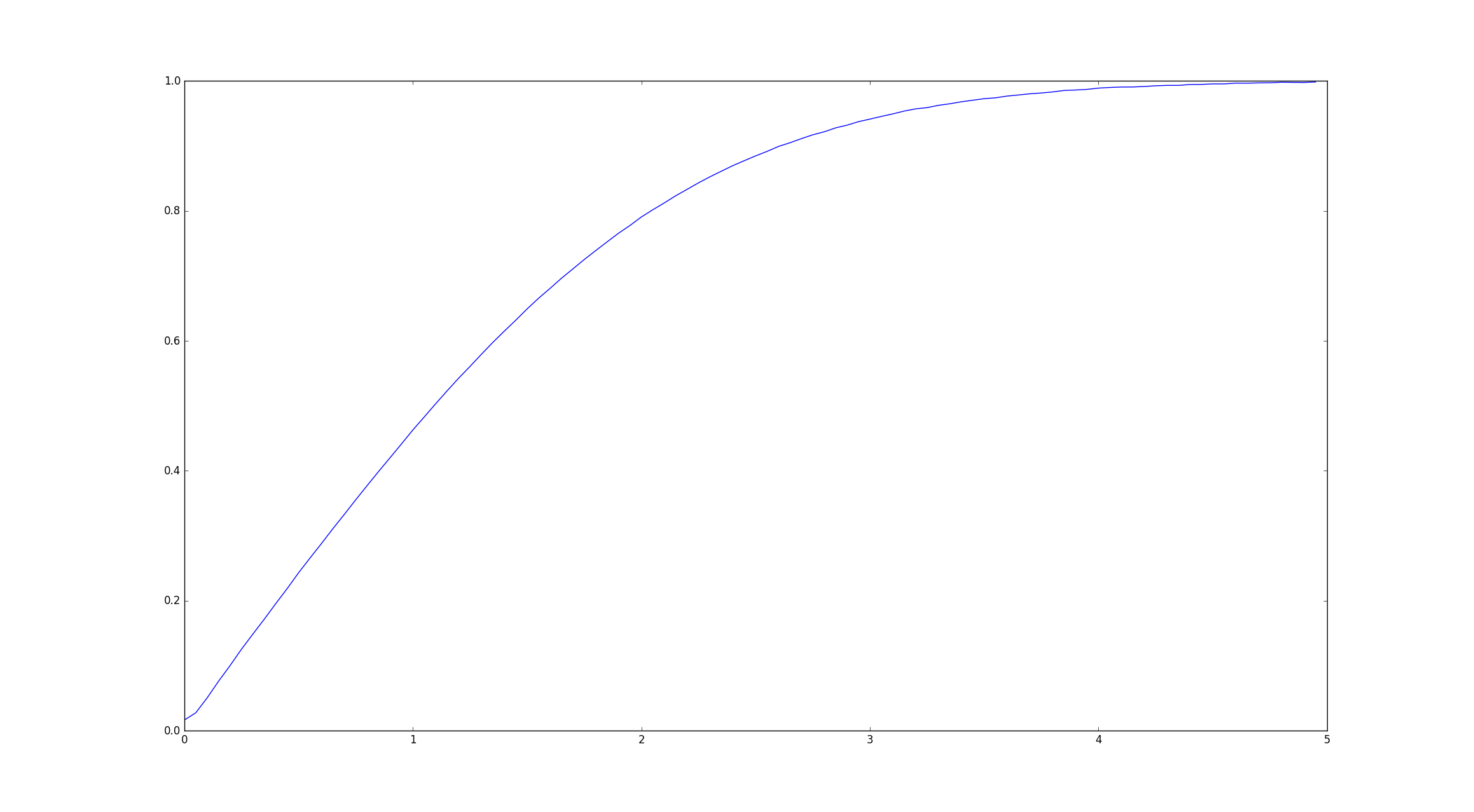}
\vspace*{-20pt}
\caption{The $H_1$ function.}
\label{fig:4}
\end{center}
\end{figure}

\subsection{Second perturbation of the Wang map: noncontinuous case}
\label{subsec-ncontM}

The resulting distribution for the example in this section is not exactly the generalised 
St.~Petersburg distribution, but it is similar to it.
We proceed as in the previous section, but this time suppose that 
$$
\xi_n = \frac12 n^{-\alpha} (1+2^{\{ \alpha \log_2 n\} }), \ n \geq 1,
$$
and set $\xi_0 = 1$. As before,  $\{ \cdot \}$ stands for the fractional part. It is easy 
to see that $\xi_n$ is strictly decreasing.
Define
\[
T(x) = 
\begin{cases} 
\frac{\xi_n-x}{\xi_n-\xi_{n+1}} \xi_n + \frac{x-\xi_{n+1}}{\xi_n-\xi_{n+1}} \xi_{n-1}, & 
\text{for }  x \in [\xi_{n+1}, \xi_n], \ n \geq 1,  \\
2x-1, & \text{ for } x \in (\frac12,1].
\end{cases}
\]
and note that $T(\xi_{n+1}) = \xi_n$ for all $n$.
It turns out that the derivative at $0$ does not exist. Indeed,
$$
\frac{T(\xi_{n+1})}{\xi_{n+1}} = \frac{\xi_n}{\xi_{n+1}}
= \frac{(n+1)^\alpha}{n^\alpha} 
\frac{1+2^{\{ \alpha \log_2 n \}} }{1+2^{\{ \alpha \log_2 (n+1) \}}}.
$$
Clearly $\frac{(n+1)^\alpha}{n^\alpha} \to 1$, but
$\frac{1+2^{\{ \alpha \log_2 n \}} }{1+2^{\{ \alpha \log_2 (n+1) \}}} $
is only close to $1$ if there is {\bf no} integer between $\alpha \log_2 n$ and
$\alpha \log_2 (n+1)$. Equivalently $n \neq \lfloor 2^{k/\alpha} 
\rfloor$ for any integer $k$. 
Thus, the sequence $(T(\xi_{n})/\xi_n)$ has two limit points, 1 and $3/2$:
\[
1 = \liminf_{x \downarrow 0} \frac{T(x)}{x} < \limsup_{x \downarrow 0} \frac{T(x)}{x} = 
\frac{3}{2}.
\]
Although the derivative at $0$ does not exist, we can still speak of the `derivative at 
$0$ along subsequences'. Indeed, if $(n_j)$ is any increasing sequence taking values in
$\N \setminus \{ \lfloor 2^{k/\alpha} \rfloor : k \in \N \}$,
then $\lfloor \alpha \log_2 n_j \rfloor =  \lfloor \alpha \log_2(n_j+1) \rfloor$.
Therefore, $\{ \alpha \log_2 n_j \} - \{ \alpha \log_2 (n_j+1) \} \to 0$, and
\begin{equation*} \label{eq-secWderiv}
\lim_{j \to \infty}
\frac{1+2^{\{ \alpha \log_2 n_j \}} }{1+2^{\{ \alpha \log_2 (n_j+1) \}}} =1.
\end{equation*}

Again, let $\tau$ be the first return time to $[1/2,1]$, and $m$ the normalised Lebesgue 
measure on $[1/2, 1]$. For $n \geq 1$, $m(\tau=n +1 )= \Delta \xi_n=\xi_n-\xi_{n-1}$, 
thus
$$
m(\tau > n)=\sum_{j\ge n}\Delta \xi_n
= \xi_n = \frac12 n^{-\alpha} \left(1+2^{\{ \alpha \log_2 n\} } \right).
$$
Thus $\tau$ satisfies~\eqref{eq-dgp-distf} with $\ell^* \equiv 1$, 
$M(x)=(1+2^{\{ \alpha \log_2 x\} })/2$, $k_n = 2^{n}$, and $c=2$. 
Again $\delta(x)$ is changed to $x$. We have 
$M(2^{1/\alpha} x) = M(x)$ for all $x \in (0,\infty)$, as required.

\subsection{Piecewise linear maps generated out of the  Fibonacci sequence}
\label{subsec-Fib}

Denote the sequence of Fibonacci numbers by
$\{ S_0, S_1, S_2, S_3, \dots \} = \{ 1,2,3,5, \dots\}$.
From Binet's formula, we get
\begin{equation}\label{eq:fibo}
S_n = \frac{1}{\sqrt{5}} ( G^{n+2} - (-G)^{-n-2}) 
=  q_0 (1 - (-1)^n G^{-2(n+2)} )\, G^n,
\end{equation}
where $G = \frac{1+\sqrt{5}}{2}$ is the golden mean  and 
$q_0 = \frac{3+\sqrt{5}}{2\sqrt 5}$.

Fix $\lambda\in (1/G,1)$, let $Y = [0,1]$ and define $T_{Y,\lambda} : Y\to Y$ by 
$T_{Y, \lambda}(0)=0$ and 
\begin{equation} \label{eq-Ty}
T_{Y,\lambda}(y) = \frac{\lambda^{n}-y}{\lambda^n-\lambda^{n+1}} \qquad
y \in (\lambda^{n+1}, \lambda^n].
\end{equation}
The map $T_{Y,\lambda}$ preserves the Lebesgue measure $m$. Given the probability 
preserving transformation $(Y, \B(Y), m , T_{Y,\lambda})$ and a measurable function 
$\tau: Y\to \N$ we construct
the Kakutani tower/map $(X, \B(X), \mu, T_\lambda)$
(see, for instance,~\cite[Chapter 5]{Aaronson} and references therein) as follows.

Let 
\begin{itemize}
\item $X := \cup_{n\geq 1} \left( \{\tau \geq n\}\times  \{n\} \right)$;
\item $\mathcal{B}(X):= 
\sigma \left(  C_n \times  \{n\}: \, 
C_n \in \B(Y) \cap \{\tau \geq n\} \  \text{for all } n\geq 1 \right)$;
\item for all $A \in ( \B(Y) \cap \{\tau \geq n \}) $,
set $\mu(A\times\{n\})=m(A)$.
\end{itemize}
Given  $T_{Y,\lambda}$ introduced in~\eqref{eq-Ty}, define the tower map
$T_\lambda:X\to X$ by
\begin{align} \label{eq-KFib}
T_\lambda(x,n) = 
\begin{cases} 
(y, n+1), & \tau(y) > n, \\
(T_{Y,\lambda}(y), 1), &\tau(y)=n.
\end{cases}
\end{align}
By construction, $T_\lambda$ preserves $\mu$. Moreover, $\tau$ is the first return time 
of $T_\lambda$ to the 
base $Y \times \{1\}$ and 
$T_\lambda^{\tau(y)}(y,1) = (T_{Y,\lambda}(y), 1)$.
In what follows, we set $\tau(y) = S_n$ for $y \in (\lambda^{n+1}, \lambda^n]$, so
\begin{equation}
\label{eq-tauFib}
m(\tau = S_n) = (1-\lambda) \lambda^n.
\end{equation}
We shall show that $m(\tau>x)$ satisfies (\ref{eq-dgp-distf}).

By~\eqref{eq-tauFib}, for $S_n \leq x < S_{n+1}$, we have
$$
m (\tau > x) = m (\tau \geq S_{n+1}) = \sum_{j \geq n+1} \mu(\tau = S_j)
= (1-\lambda)\sum_{j \geq n+1}\lambda^j = q_n S_n^{-\alpha},
$$
where $\alpha = -\log_G \lambda$ and
\begin{equation*} \label{eq-qn}
q_n = \lambda 
\left(\frac{3+\sqrt{5}}{2\sqrt{5}}\right)^\alpha  (1-(-1)^n G^{-2(n+2)})^\alpha.
\end{equation*} 
We can easily see that  
\begin{equation*} \label{eq:limq}
\lim_{n \to \infty} q_n = \lambda \left(\frac{3+\sqrt{5}}{2\sqrt{5}}\right)^\alpha
=: q_\infty > 0.
\end{equation*}
Moreover, since $\lambda\in (1/G,1)$, we have $\alpha\in (0,1)$. 
Since
$$
S_n^{-\alpha} = x^{-\alpha} \left( \frac{x}{S_n} \right)^\alpha
= x^{-\alpha} G^{\alpha \log_G \frac{x}{S_n} },
$$
we have
\begin{equation} \label{eq-tauFib2}
m(\tau > x) = q_n x^{-\alpha} G^{\alpha \log_G \frac{x}{S_n} }
\qquad\text{ for } S_n \leq x < S_{n+1}.
\end{equation}
Equipped with the above, we state

\begin{prop}\label{prop-semfib} 
The tail of $\tau$ in \eqref{eq-tauFib2} satisfies
\eqref{eq-dgp-distf} with
$c = G^\alpha$, $k_n = \lfloor G^{\alpha n } \rfloor$, $\ell(x) \equiv 1$, 
$A_k = k^{1/\alpha}$, and $M(x) = q_\infty G^{\alpha \{ \log_G (x/q_0) \}}$; i.e.
$$
m(\tau > x) = x^{-\alpha} \left( M(x) + h(x) \right),
$$
where $\lim_{n\to\infty}h(A_{k_n} x)=0$ whenever $x$ is continuity point of $M$.
\end{prop}

\begin{proof}
First recall again that if $\ell \equiv 1$ and $k_n = \lfloor c ^n \rfloor$ then 
$\delta(x)$ in \eqref{eq-dgp-distf} can be replaced with $x$.

By \eqref{eq-tauFib2} we may write
\[
m( \tau > x ) = x^{-\alpha}
\left( q_\infty G^{\alpha \{ \log_G \frac{x}{q_0} \}} + h(x) \right)
\]
with
\begin{equation} \label{eq:Fib-h}
h(x) = q_\infty \left( G^{\alpha \log_G \frac{x}{S_n}} - G^{\alpha \{ \log_G 
\frac{x}{q_0} \}} \right) + (q_n - q_\infty ) G^{\alpha \log_G \frac{x}{S_n}}, \quad
x \in [S_n, S_{n+1}). 
\end{equation}
Therefore, we only have to prove that the error function $h$ satisfies the necessary 
properties in \eqref{eq-dgp-distf}; i.e.~$\lim_{n \to \infty} h(A_{k_n} x) = 0$, 
whenever 
$x \in C_M$.

In \eqref{eq:Fib-h} the second summand converges to 0 as $x \to \infty$. So we have to 
prove that for any $x \in C_M$
\[
\lim_{n \to \infty} \log_G \frac{x \lfloor G^{\alpha n} \rfloor^{1/\alpha}}{S_{m}} =
\left\{ \log_G \frac{x}{q_0} \right\},
\]
where $m= m_n$ is uniquely determined by $S_{m} \leq x \lfloor G^{\alpha n} 
\rfloor^{1/\alpha} < S_{m +1}$. This follows easily from \eqref{eq:fibo} and that $x \in 
C_M$ if and only if $\{ \log_G (x / q_0) \} >0$.
\end{proof}

Proposition~\ref{prop-semfib} shows that the return times corresponding to the map 
$T_\lambda$  satisfy tail condition~\eqref{eq-dgp-distf}. Hence, a semistable law for 
$(T_{Y,\lambda},\tau)$ holds along  $k_n = \lfloor G^{\alpha n}\rfloor$.
Thus, Theorem \ref{thm-merge} applies to $T_\lambda$, giving the distributional behaviour 
of $S_{n_r}$, for suitable subsequences $n_r$.
Moreover, since $k_n = \lfloor G^{\alpha n}\rfloor$,   Proposition 
\ref{prop:renewal} (and thus~\eqref{eq-avMS}) holds.

\section{A process, which is not isomorphic to a Markov renewal chain}
\label{sec-dets}

In this section we show that  (the conclusion of) Theorem~\ref{thm-merge} and 
Theorem~\ref{prop-genav} apply to infinite measure preserving systems  that are  not 
isomorphic to Markov renewal chains. To fix terminology,  in 
Subsection~\ref{subsec-defdyn} we provide a 
simple smooth version of the renewal shift~\eqref{eq-firstW} 
considered in Subsection~\ref{subsec-contM}; this is given by
 the family of smooth Markov maps $f_\eps:[0,1]\to [0,1]$ (defined in~\eqref{eq-fepsdet}) 
with indifferent fixed point at $0$.
In Subsection~\ref{subsec-ind}, we note that the first return time to a subset of $[0,1]$ 
satisfies the tail condition~\eqref{eq-dgp-distf} and verify that the corresponding 
induced family of maps  satisfy good distortion properties.
The latter allows to conclude in Subsection~\ref{subsec-spectral} that the main 
functional analytical properties of the induced map hold and 
in Subsection~\ref{subsec-rediid} we justify that the conclusion of 
Theorem~\ref{thm-merge} holds for $f_\eps$.
Using the same functional analytic properties in Subsection~\ref{subsec-tr} we  show that 
Theorem~\ref{prop-genav} applies, obtaining the exact sequences and scaling for the 
convergence of the average transfer operator~\eqref{eq-unif},
uniformly on compacts of $(0,1]$.
Finally, we mention that, although the results of this sections are in terms of a  simple 
example, the same arguments apply to  dynamical systems with infinite measure satisfying  
tail condition~\eqref{eq-dgp-distf} along with properties (A1) and (A2) 
stated below. For a discussion of our results on infinite measure preserving systems we refer to
Subsection~\ref{subsec-rediid}.

\subsection{A smooth version of the example in Subsection~\ref{subsec-contM}}
\label{subsec-defdyn}

For fixed $\alpha\in (0,1)$, $c>1$ and $\eps$ small enough, we define $(\xi_n)_{n \geq 0}$ 
as in Subsection~\ref{subsec-contM}, that is 
$\xi_n = \frac{1}{2} n^{-\alpha} \left(1+2\eps \sin(\frac{2\pi \alpha \log n}{\log 
c})\right)$, $n\ge 2$ and $\xi_0 = 1$, $\xi_1 = \frac12$. We recall that, 
as clarified in Subsection~\ref{subsec-contM}, $\xi_n$  is decreasing.
In what follows, out of  $(\xi_n)_{n \geq 0}$  we define a smooth map $f_\eps: [0,1]\to 
[0,1]$ via the map $f_{\eps,n}$ defined below.

A lengthy computation based on \eqref{eq-Deltaxi} (see Appendix~\ref{subsec-lemma2} for 
details) shows that 
there exist $\alpha_n = (1+\alpha+O(n^{-1}))(1+O(\eps))$ and $r_n = O(1)$ such that 
\begin{equation}\label{eq:alphan}
\frac{\Delta \xi_{n-2}}{\Delta \xi_{n-1}} = \frac{\xi_{n-2}-\xi_{n-1}}{\xi_{n-1}-\xi_n} 
= 1+\frac{\alpha_n}{n}  +\frac{r_n}{n^2} + R_n,
\end{equation}
and $R_n = O(n^{-3})$.
Let $f_{\eps, n}:[\xi_n,\xi_{n-1}] \to [\xi_{n-1},\xi_{n-2}]$, $n\ge 2$, be defined by
$$
f_{\eps,n}(x) = \frac{A_n}{2} \frac{(x-\xi_n)^2}{\xi_{n-1}-\xi_n}
+ (1+\frac{\alpha_n}{n} + B_n) (x-\xi_n)  + \xi_{n-1},
$$
Here $A_n$ and $B_n$ will be chosen appropriately.
We note that
$$
f_{\eps,n}(\xi_n) = \xi_{n-1}, \qquad f_{\eps,n}(\xi_{n-1}) = \xi_{n-2}.
$$
The first is automatic, and the second follows provided
\begin{equation}\label{eq:An1}
\frac{r_n}{n^2} + R_n = B_n + \frac{A_n}{2}.
\end{equation}
The derivative is
$$
f'_{\eps,n}(x) = \frac{x-\xi_n}{\xi_{n-1}-\xi_n} A_n + 1+\frac{\alpha_n}{n} + B_n. 
$$
We note that 
$$
f'_{\eps,n}(\xi_n) = 1+\frac{\alpha_n}{n} + B_n, \qquad 
f'_{\eps,n}(\xi_{n-1}) = 1+\frac{\alpha_{n-1}}{n-1} + B_{n-1}. 
$$
provided that (the first equation above is automatic)
\begin{equation}\label{eq:An2}
A_n = \frac{\alpha_{n-1}}{n-1} - \frac{\alpha_n}{n} + B_{n-1} - B_n, \quad n \geq 3.
\end{equation}
Solving for $B_n$ from \eqref{eq:An1} and \eqref{eq:An2}, we get the recursive formula
\begin{equation*}\label{eq:Bn}
 B_n = - B_{n-1} + \frac{2r_n}{n^2} + 2R_n - 
 \frac{\alpha_{n-1}}{n-1} + \frac{\alpha_n}{n} 
=: -B_{n-1} + q_n,
\end{equation*}
for $q_n = \frac{2r_n}{n^2} + 2R_n - \frac{\alpha_{n-1}}{n-1} + \frac{\alpha_n}{n} = 
O(n^{-2})$.
Working out the recursion, we get
$$
B_n = (-1)^n \left( B_2 + \sum_{i=3}^n (-1)^i q_i \right), \quad i \geq 3,
$$
for arbitrary $B_2 \in \R$.
Clearly, 
$(B_n)_{n \geq 0}$ is bounded. Choose $B_2 = -\sum_{i=3}^\infty 
(-1)^i q_i$.
Then 
\begin{equation}\label{eq:Bn2}
B_n = (-1)^{n+1} \sum_{i=n+1}^\infty (-1)^i q_i, \quad n \geq 3.
\end{equation}
Note that
$$
q_n - q_{n+1} = \left( \frac{2r_n}{n^2} - \frac{ 2r_{n+1} }{ (n+1)^2 } \right) + 
2(R_n-R_{n+1}) -
 \left( \frac{\alpha_{n-1}}{n-1} - \frac{2\alpha_n}{n} - \frac{\alpha_{n+1}}{n+1} \right).
$$
One can check that the $\alpha_n$'s and $r_n$'s change so slowly with $n$ that 
$q_n-q_{n+1} = O(n^{-3})$; see Appendix \ref{subsec-lemma2}.
Rewriting \eqref{eq:Bn2}, we have 
$$
B_n = (-1)^{n+1} \sum_{k=1}^\infty \left[ (-1)^{n+2k-1} q_{n+2k-1} + (-1)^{n+2k} q_{n+2k} 
\right]
= \sum_{k=1}^\infty q_{n+2k-1} - q_{n+2k} = O(n^{-2}).
$$
By \eqref{eq:An2}, $A_n = O(n^{-2})$. Therefore $A_n$ and $B_n$ can be chosen 
appropriately.

Define the map $f_\eps:[0, 1] \to [0,1]$,
\begin{align}
\label{eq-fepsdet}
 f_\eps(x) =\begin{cases} 
 f_{\eps,n}(x), & x \in [\xi_n, \xi_{n-1}], n\ge 2 \\
 2x-1, & x>\frac 12.
\end{cases}
\end{align}
By~\eqref{eq-Deltaxi}, we have that 
$f_\eps$ is differentiable at $0$ (from the right) and $0$ is an indifferent fixed point.

\subsection{Induced map, tail distribution, infinite invariant measure
}
\label{subsec-ind}

Let $\tau$ be the 
first return time of $f_\eps$ to $Y=[\frac 12, 1]$ and  define the induced map 
$F_\eps=f_\eps^\tau$. Note that
\begin{align*}
F_\eps:Y \to Y, \qquad
F_\eps(x) = \begin{cases}
   2x-1  & \text{ if } x \in [\frac34, 1],\\      
   f_\eps^{n-1}(2x-1)  & \text{ if } x \in [\frac{\xi_n+1}{2}, \frac{\xi_{n-1}+1}{2}), 
n\ge 2,\\      
  \frac 12  & \text{ if } x=\frac 12,
\end{cases}
\end{align*}
has onto branches. In fact, as clarified below, the induced map $(Y,\mathcal A(Y), 
F_\eps, \gamma)$, where $\mathcal A(Y)$ the Borel sigma algebra on $Y$ and 
$\gamma=\{[\frac{\xi_n+1}{2}, \frac{\xi_{n-1}+1}{2})\}_{n\ge 1}$,  is \emph{Gibbs--Markov}
(for complete definitions see~\cite{AaronsonDenker01}, ~\cite[Chapter 4]{Aaronson}).

First, it follows from the above definition of $F_\eps$ and $\gamma$ that each element of 
$\gamma$ is mapped bijectively onto a union of partition elements. 

Next, differentiating $\Delta \xi_n$ in~\eqref{eq-Deltaxi}  w.r.t.~$n$ gives that it is 
strictly decreasing in $n$, so
$\Delta \xi_{n-2}/\Delta \xi_{n-1} > 1$ and $f'_{\varepsilon,n}(x) > 1$ for all $x \in 
[\xi_n, \xi_{n-1}]$.
Moreover, $f'_\varepsilon(x) = 2$ on $(\frac12, 1]$ by~\eqref{eq-fepsdet}.
By the chain rule we get $F'(x) \geq 2$ as well for every 
$x \in (1/2,1] \setminus\{ (\xi_n+1)/2 \}_{n\ge 2}$.
Thus,  $F_\eps$ is \emph{expanding} on each element of its Markov partition $\gamma$.
As a consequence, for every two points $x,y$ there exists  $n \geq 0$ such that $F^n(x)$ 
and $F^n(y)$ lie
in different elements of $\gamma$. Therefore, the atoms of the partition 
$\bigvee_{n=0}^\infty F^{-n}\gamma$
are points, which implies that $\gamma$ is a \emph{generating partition} for $F_\eps$ 
(that is, $\sigma(\{F_\eps^{-n}\gamma : n \ge 0\}) = \mathcal A(Y)$).

Also, by Lemma~\ref{lemma-distfn} in 
Appendix~\ref{subsec-lemma2}, 
$F_\eps$ is piecewise $C^2$ and the \emph{distortion condition} 
$\frac{|F_\eps''|}{(F_\eps')^2}<\infty$ 
holds.  The above verified properties, Markov generating partition, expansion and 
distortion conditions guarantees that  $(Y,\mathcal A(Y), F_\eps, \gamma)$ is 
Gibbs--Markov. For the fact that the above distortion condition can be used as part of 
the definition of a Gibbs--Markov map, see~\cite[Example 2]{AaronsonDenker01} 
and~\cite[Chapter 4]{Aaronson}.

Since  $(Y,\mathcal A(Y), F_\eps, \gamma)$ is 
Gibbs--Markov, $F_\eps$  preserves a measure $\mu_Y$ with density 
$h=\frac{\dd \mu_Y}{\dd m}$, with $m$ being the normalised Lebesgue measure on $Y$, 
bounded from above and below and $h\in C^2(Y)$ (see~\cite{AaronsonDenker01} 
and~\cite[Chapter 4]{Aaronson}). Thus,
\begin{align}
\label{eq-tailFep}
\nonumber\mu_Y(\tau>n)&=\int_{\frac 12}^{\frac{\xi_n+1}{2}}h(x)\, \dd m(x) 
=h(1/2)\xi_{n}(1+o(1))\\
&= \frac12 h(1/2)n^{-\alpha} 
\left(1+2\eps \sin \left( \frac{2\pi \alpha \log n}{\log c} \right) \right) (1+o(1)).
\end{align}
An $f_\varepsilon$-invariant measure $\mu$ can be 
obtained by pulling back: 
$$
\mu(A) = \sum_{n \geq 0} \mu_Y( f^{-n}(A) \cap \{ \tau > n\} )
$$
for every Borel measurable set $A$. 
Then $\mu([0,1]) = \sum_{n \geq 0} \mu_Y(\{ \tau > n\} ) = \infty$,
so $\mu$ is infinite.

Similar to  Subsection~\ref{subsec-contM}, we let $M(x) = \frac12 h(1/2)(1 + 2\eps \sin( 
\frac{2\pi \alpha \log x}{\log c}))$, define $k_n = \lfloor c^n \rfloor$ and set 
$A_k = k^{1/\alpha}$, so $A_{k_n} = \lfloor c^{n}\rfloor^{1/\alpha}$.
Thus, $\tau$ satisfies \eqref{eq-dgp-distf} with $\ell \equiv 1$.

\subsection{Functional analytical properties of the induced map}
\label{subsec-spectral}

Since it is not going to play a role, throughout the remaining of this section we suppress 
the dependency of the induced map
on $\eps$ and set $F:=F_\eps$.

Let $R:L^1(\mu_Y)\to L^1(\mu_Y)$ be the transfer operator associated with $F: Y\to Y$
defined by $\int_{Y} R^n v \cdot  w\, \dd \mu_Y=\int_{Y}  v \cdot  w\!\circ\!F^n\, \dd 
\mu_Y$, for 
all $v\in L^1(\mu_Y)$, $w\in L^\infty(\mu_Y)$ and $n\geq 1$.

To  apply the inversion procedure (duality rule) to $f_\eps$ and thus verify that the  
conclusion of Theorem~\ref{thm-merge} holds, we first need to understand the 
distributional behaviour of $Z_n=\sum_{j=0}^{n-1}\tau \! \circ \! F^j$.
To do so we recall  the classical procedure of establishing limit laws
for  Markov maps 
with good distortion properties, as developed by Aaronson and Denker 
in~\cite{AaronsonDenker01}; see also the survey paper by Gou\"ezel \cite{Gou15}. This 
means  to relate  Fourier transforms to perturbed transfer operators.
A rough description of the procedure for showing  convergence in distribution of $Z_n$ 
when appropriately scaled with some 
norming sequence $a_n$ goes as follows.
For $\theta\in \R$,
$$
\E_{\mu_Y}(e^{ i \theta a_n^{-1} Z_n})= \int_Y e^{ i\theta a_n^{-1} Z_n}\, \dd 
\mu_Y=
\int_Y R^n (e^{ i\theta a_n^{-1} \tau})\, \dd \mu_Y.
$$
The above formula says that understanding of the Fourier transform
$\E_{\mu_Y}(e^{ i \theta a_n^{-1} \tau_n})$, for $|\theta|/a_n$ sufficiently small,  comes 
down to understanding the behaviour of the perturbed transfer operator 
$$
\hat R(\theta)v:=R(e^{i\theta \tau}v ), \quad v\in L^1(\mu_Y).
$$

From Subsection~\ref{subsec-ind}, we know  that $(Y,\mathcal A(Y), F_\eps, \gamma)$ is 
Gibbs--Markov.
We recall some 
properties of $R$ in the Banach space  $\B$ of bounded 
piecewise (on each element of $\gamma$) H{\"o}lder functions 
with $\B$ compactly embedded in $L^\infty(\mu_Y)$.
The norm on $\B$ is 
$\| v \|_\B = |v|_\theta + |v|_\infty$,
where $|\cdot|_\infty$ is the usual sup norm, and
$|v|_\theta = \sup_{a \in \gamma} \sup_{x \neq y \in a} |v(x)-v(y)| / 
d_\theta(x,y)$, where $d_\theta(x,y) = \theta^{s(x,y)}$ for some $\theta\in (0,1)$,
and $s(x,y)$ is the separation time, i.e.~$s(x,y)$ is the minimum $n$ such  that 
$F^n(x), F^n(y)$ lie in different partition elements.

By Theorem 1.6 in \cite{AaronsonDenker01},
\begin{itemize}
\item[(A1)] $1$ is a simple, isolated eigenvalue in the spectrum of $R$, when viewed as 
an operator acting on $\B$.
\end{itemize}

Set $R_n v=R (1_{\{\tau=n\}}v)$,  $n\ge 1$, $v\in L^1(\mu_Y)$. Note that 

\[
\hat R (\theta) ( v) = R \left( \sum_{n=1}^\infty e^{i \theta n} 1_{\tau = n} v \right)
= \sum_{n=1}^\infty e^{i \theta n} R( 1_{\tau=n} v) = \sum_{n=1}^\infty e^{i \theta n } 
R_n(v),
\]
which says that $\hat R(\theta)=\sum_{n\ge 1} R_n  e^{in\theta}$ (in particular, $\hat 
R(0)=\sum_{n\ge 1} R_n$).

As shown in~\cite[Lemma 8 and formula (8)]{Sarig02}, which works with 
$R_n(v) = 1_Y L^n(1_{\tau =n } v)$,
\begin{itemize}
\item[(A2)] For $n\ge 1$, $R_n:\B\to\B$ is a bounded linear operator with  $\|R_n\|\le 
C\mu_Y(\tau=n)$, for some $C>0$.
\end{itemize}

By~\eqref{eq-tailFep}, $\mu_Y(\tau>n)\ll n^{-\alpha}$. This together with the 
definition
of $\hat R(\theta)$ and (A2) implies that $\|\hat R(\theta)-\hat R(0)\|\leq 
C|\theta|^\alpha$, for some $C>0$ (see, for instance,~\cite[Proposition 2.7]{MT12}). This 
together with (A1) implies that there exists $\delta>0$ and a $C^\alpha$ family
of eigenvalues $\lambda(\theta)$ well defined in $B_\delta(0)$ with $\lambda(0)=1$.

\begin{lemma}
\label{lemma-eigF} Given that (A1) and (A2) hold for the induced map $F$, we have that 
\begin{itemize}
\item[a)] $\lambda(\theta)=\int_Y e^{i\theta\tau}\, \dd \mu_Y+O(\theta^{2\alpha})$  as 
$\theta\to 0$;
\item[b)] Let $a_n\to\infty$.  Then for all $\theta$ such that  $\theta<a_n\delta$ (so, 
$\lambda(\theta a_n^{-1})$ is well defined) and for some $\sigma\in (0,1)$,
$$\E_{\mu_Y}(e^{ i\theta a_n^{-1} Z_n})
=\lambda(\theta a_n^{-1})^n(1+o(1))+O(\sigma^n).
$$
\end{itemize}
\end{lemma}

\begin{proof} Given (A1) and that $\hat R(\theta)$ is $C^\alpha$, we have: item a) is 
contained in, for instance,~\cite[Proof of Lemma A.4]{MT13};
item b) follows as in~\cite[Proof of Theorem 6.1]{AaronsonDenker01}.~\end{proof}

\subsection{The Darling--Kac law along subsequences}
\label{subsec-rediid}


As already mentioned in the introductory paragraph of the present section, here we phrase 
our results Propositions~\ref{prop-DKniid} and~\ref{prop-avop} and in terms of 
example~\eqref{eq-fepsdet}, but, as obvious from the corresponding proofs, the same 
arguments apply to  dynamical systems with infinite measure satisfying  
tail condition~\eqref{eq-dgp-distf} along with properties (A1) and (A2)  above (which 
could hold in a different function space $\B$). Proposition~\ref{prop-DKniid}
gives a Darling--Kac law along subsequences for such non iid systems and all involved 
notions have been clarified in previous sections. Proposition~\ref{prop-avop} gives 
\emph{uniform dual ergodicity along subsequences} and we clarify this terminology below.

Let $(X, \mu)$ be an infinite measure space and $T:X\to X$ be a conservative
measure preserving transformation with transfer operator  
$L: L^1(\mu)\to L^1(\mu)$, 
$\int_{X} L^n v \cdot w\, \dd \mu=\int_{X}  v \cdot w\!\circ\! f^n\, \dd \mu$,
for all $v\in L^1(\mu)$, $w\in L^\infty(\mu)$ and $n\geq 1$. 
The transformation $T$ is \emph{pointwise dual ergodic} if there exists a positive
sequence $a_n$ such that $a_n^{-1} \sum_{j=0}^n L^j v\to \int_ X v\,  \dd\mu$ a.e. as 
$n\to\infty$, for all  $v\in L^1 (\mu)$. If, furthermore, there exists $Y\subset X$ with 
$\mu(Y)\in (0,\infty)$ such that
$a_n^{-1} \sum_{j=0}^n L^j 1_Y\to \mu(Y)$, uniformly on $Y$ , then $Y$ is referred to as 
a \emph{Darling--Kac set} (see~\cite{Aaronson} for further background) and we refer to 
$T$ as \emph{uniformly dual ergodic}.  It is still an open question whether every 
pointwise dual ergodic transformation has a Darling--Kac set. However, it is desirable
to prove pointwise dual ergodicity by identifying Darling--Kac sets, as this facilitates 
the proof of several strong properties for  $T$; in particular, the existence of a 
Darling--Kac set along regular variation for
the return time to this set implies that $T$ satisfies a Darling--Kac law  
(see~\cite{Aaronson, ThalerZweimuller06} and references therein).
Furthermore, in the presence of regular variation of the return time to `good' sets, 
Melbourne and Terhesiu~\cite{MT13} have obtained uniformly dual ergodic theorems with 
remainders (in some cases, optimal remainders).

When regular variation is violated is still possible to obtain \emph{uniform dual 
ergodicity along subsequences} (and thus, pointwise dual ergodicity along subsequences);
this is the content of  Proposition~\ref{prop-avop} and the identification of the allowed 
class of subsequences is, of course, the main novelty. 
We do not know whether a Darling--Kac law along subsequences  can be derived directly 
from uniform dual ergodicity along subsequences; similarly to the regularly varying case, 
this would require exploiting the method of moments
and our methods are not applicable.

Throughout the rest of this section, we let $f=f_\eps$,  $F=F_\eps=f^\tau$  and recall 
that $\mu_Y$ and $\mu$ are $F$ and $f$, respectively,
invariant. We recall from  Subsection~\ref{subsec-ind} that $\tau$ satisfies 
\eqref{eq-dgp-distf} with $\ell \equiv 1$
and using the same notation, we let $k_n = \lfloor c^n \rfloor$ and set 
$A_{k_n} = k_n^{1/\alpha}$.

Using Lemma~\ref{lemma-eigF}, in this paragraph we clarify that $\tau$ is in the domain 
of geometric partial attraction of  a semistable law.
As a consequence,  the conclusion of Theorem~\ref{thm-merge} holds  for $f$, which we 
restate below in full generality. 
\begin{prop}
\label{prop-DKniid}
\begin{itemize}
\item[(i)]  There exists a semistable random variable $V$ (as defined 
in~\eqref{eq:semistable-chf-df}) such that for any $x>0$
and for any probability measure $\nu_Y \ll \mu_Y$, 
\[
\lim_{n \to \infty} \nu_Y \left(A_{k_n}^{-1} Z_{k_n}\le x\right)
= \P \left( V\le x\right).
\]
Moreover, given $\gamma(\cdot)$ as in~\eqref{eq:def-gamma},
\[
\lim_{r \to \infty} \nu_Y\left(A_{n_r}^{-1} Z_{n_r}\le x\right)
=\P \left(V_\lambda\le x\right),
\]
whenever $\gamma({n_r})\stackrel{cir}{\to} \lambda$, where  $V_\lambda$ is a semistable 
random variable as defined in~\eqref{eq:semistable-chf-df-lambda}.
\item[(ii)] Let $S_n = S_n(1_Y) = \sum_{j=0}^{n-1} 1_Y \!\circ\! f^j$. Suppose that 
$\gamma(a_{n_r}) \stackrel{cir}{\to} \lambda \in (c^{-1}, 1]$. Then for any $x > 0$, and 
for any probability measure $\nu\ll\mu$,
\begin{equation*} 
\lim_{r \to \infty} \nu ( S_{n_r}(v)  / a_{n_r} \leq x ) = 
\P \left( (V_{h_\lambda(x)} )^{-\alpha} \leq x \right)
=H_\lambda(x),
\end{equation*}
where $h_\lambda(x) = \frac{\lambda x}{c^{\lceil \log_c (\lambda x) \rceil }}$.
More generally, 
\[
\lim_{n \to \infty} \sup_{x > 0} \left| \nu ( S_n  \geq a_n x) -  
\P ( V_{\gamma(a_n x)} \leq x^{-1/\alpha}) \right| = 0.
\]

Moreover, the asymptotic behaviour of the distribution $H_\lambda$ at $\infty$ and $0$ 
are as given in Lemma~\ref{lemma:tail-est} and Theorem~\ref{thm:Hat0}, respectively.
\end{itemize}
\end{prop}

We remark that (ii) holds true for 
$S_n(v) = \sum_{j=0}^{n-1} v \! \circ \! f^j$ 
for any $v \in L^1(\mu)$ such that $\int v \, \dd \mu \neq 0$. Indeed,
write
\[
\frac{S_n(v)}{a_{n}} = \frac{S_n(v)}{S_n} \frac{S_n}{a_n},
\]
and note that by Hopf's ratio ergodic theorem
(see, for instance,~\cite[Ch.2]{Aaronson} and~\cite[Section 
5]{ThalerZweimuller06}; see also~\cite{Zweimuller04} for a 
short proof of this theorem) the first factor converges a.s.~as $n \to \infty$ to
$\int v \dd \mu / \mu(Y)$.

\begin{proof}{\bf (i)}
Recall the map $T:=T_\eps$ defined by~\eqref{eq-firstW} in Subsection~\ref{subsec-contM}; 
as noted there, $T$ is isomorphic to a renewal shift .
Let $\tilde\tau$ be the first return time of $T$ to $Y=[1/2,1]$ and set 
$T_Y=T^{\tilde\tau}$. Let $m$ be the normalised Lebesgue measure on $Y$ and note that 
by~\eqref{eq-tailFep},
\begin{equation}
\label{eq-tauFiid}
\mu_Y(\tau>n)=h(1/2)m(\tilde\tau>n)(1+o(1)).
\end{equation}
Lemma~\ref{lemma-eigF} a), \eqref{eq-tauFiid}, and Corollary 1 in \cite{PK2} 
(it remains true for characteristic functions)
imply that as $\theta \to 0$
\begin{equation}
\label{eq-rediid}
1-\lambda(\theta)=h(1/2)\int_Y (1-e^{i\theta\tilde\tau})\, \dd m(1+o(1))=
h(1/2)(1-\E_m( e^{i\theta\tilde\tau}))(1+o(1)).
\end{equation}
Set $\tilde Z_n=\sum_{j=0}^{n-1}\tilde\tau \!\circ\! T_Y^j$.  Let $a_n\sim 
n^{1/\alpha}$ (so, 
$a_n$ satisfies~\eqref{eq:bn}). Formula \eqref{eq-rediid} together with 
Lemma~\ref{lemma-eigF} b) implies that
\begin{equation} \label{eq:charfc-compare}
\E_{\mu_Y}(e^{ i\theta a_n^{-1} Z_n})=
\left[ \E_m( e^{i\theta a_n^{-1}\tilde Z_n}) \right]^{h(1/2) }
(1+o(1))+O(\sigma^n),
\end{equation}
for some $\sigma\in (0,1)$.  
As in Subsection~\ref{subsec-contM}, 
$\lim_{n \to \infty} m\left(A_{k_n}^{-1}\tilde Z_{k_n}\le x\right)=
\P \left( \tilde V\le x\right)$
and, given $\gamma_x$ as in~\eqref{eq:def-gamma},
$\lim_{r \to \infty} m\left(A_{n_r}^{-1}\tilde Z_{n_r}\le x\right)=
\P \left(\tilde V_\lambda\le x\right)$, whenever 
$\gamma({n_r})\stackrel{cir}{\to} \lambda$. Therefore, the characteristic 
functions converge, thus by \eqref{eq:charfc-compare}
\[
\lim_{n \to \infty} \E_{\mu_Y} \left( e^{i \theta a_{k_n}^{-1} Z_{k_n}} \right)
= \left( \E e^{i \theta \tilde V} \right)^{h(1/2)},
\]
and similarly for $n_r$. We also see that the limit in the non-iid case is a 
convolution power of the limit in the iid case.
Thus (i) with $\nu_Y=\mu_Y$ follows. The statement for general $\nu_Y$ 
follows by~\cite[Proposition 4.1]{ThalerZweimuller06} (see also first sentence 
under Proposition 4.1 in ~\cite{ThalerZweimuller06} for further references).

{\bf (ii)} The statement for $\nu=\mu$ follows from item (i) and 
the duality argument used in the proof of Theorem~\ref{thm-merge}. The 
statement for $\nu\ll\mu_Y$ follows by~\cite[Proposition 
4.1]{ThalerZweimuller06}.
\end{proof}

\subsection{Asymptotic behaviour of the average transfer operator: uniform dual 
ergodicity along subsequences}
\label{subsec-tr}

Recall that $\mu, \mu_Y$  are $f$ and $F$, respectively, invariant.
Let $L:L^1(\mu)\to L^1(\mu)$ be the 
transfer operator associated with $f$.
Recall that $\B$ is the function space under which (A1) and (A2) hold
and that $\tau$ satisfies \eqref{eq-dgp-distf} with $\ell \equiv 1$
and $k_n = \lfloor c^n \rfloor$. We also recall the class $\mathcal{P}_{r,\rho}$ of 
log-periodic functions introduced
in~\eqref{eq:def-P} and let $C_p$ be the set of continuity points of 
$p\in\mathcal{P}_{r,\rho}$. Here we show that 
Theorem~\ref{prop-genav} applies to $f$ and, as a consequence obtain:

\begin{prop}
\label{prop-avop} There exists
$p\in\mathcal{P}_{r,\alpha}$ such that for any $z\in C_p$ and for any H{\"o}lder function 
 $v:[0,1]\to\R$, supported on a compact set of $(0,1]$,
\[
\lim_{n \to \infty}
\frac{\sum_{j=0}^{[ c^{n/\alpha}z]} L^j v}{ c^n}=z^\alpha\, 
p(z)\int_{[0,1]} v\,  \dd \mu,
\]
uniformly on compact sets of $(0,1]$.
\end{prop}

To show that Proposition~\ref{prop-avop} follows from Theorem~\ref{prop-genav} and 
Lemma~\ref{lemma-hatT} below (which verifies the assumption of Theorem~\ref{prop-genav})
we recall  the language of operator renewal sequences, introduced
in the context of finite measure  dynamical systems
by Sarig~\cite{Sarig02} and Gou{\"e}zel~\cite{Gouezel04} and exploited in in the context 
of infinite measure  dynamical systems
by Melbourne and Terhesiu~\cite{MT12, MT13} and Gou{\"e}zel~\cite{Gouezel11}. The proof 
of  Proposition~\ref{prop-avop} is provided at the end of this subsection.

Recall the notation used in Subsection~\ref{subsec-spectral}:
$R_n v=R(1_{\{\tau=n\}}v)$, $n\ge 1$ and define the operator sequences
\begin{equation*}
T_n v=1_YL^n (1_Y v),\enspace n\ge 1, \ T_0=I.
\end{equation*}
We note that $T_n$ corresponds to general returns to $Y$ and $R_n$ corresponds to first 
returns to $Y$.   The relationship $T_n=\sum_{j=1}^n T_{n-j}R_j$
generalises the renewal equation for scalar renewal sequences 
(see~\cite{Feller, BGT} and references therein).

For $s>0$, define the operator power series
$\hat T(e^{-s}),\hat R(e^{-s}):\B\to\B$ by
\begin{equation}\label{eq:series}
\hat T(e^{-s})=\sum_{n\ge 0} T_n e^{-sn}, \qquad \hat R(e^{-s})=\sum_{n\ge 1} R_n e^{-sn}
\end{equation}
Working with $e^{-s}$, instead of $i\theta$ in Subsection~\ref{subsec-spectral},
we have $\hat R(e^{-s})v=R(e^{-s\tau} v)$. The relationship $T_n=\sum_{j=1}^n T_{n-j}R_j$ 
together with~\eqref{eq:series}
implies that for all $s>0$,
$$
\hat T(e^{-s})=(I-\hat R(e^{-s}))^{-1}.
$$
We note that under (A1) and  (A2), $(I-\hat R(e^{-s}))^{-1}$ is well defined for $s$ in a 
neighbourhood of $0$.

The next result below gives the asymptotic of $\hat T(e^{-s})$, as $s\to 0$, as 
required 
for the application of Theorem~\ref{prop-genav}. We recall from~\eqref{eq-tailFep} that
$\mu_Y(\tau>n)=n^{-\alpha} M(n)(1+o(1))$, where $M(x) =h(1/2)\frac 12 (1 + 2\eps \sin( 
\frac{2\pi \alpha \log x}{\log c}))$.

\begin{lemma}
\label{lemma-hatT}
For $\rho>0$, let $\mathrm{A}_\rho, \mathrm{B}_\rho$ be the operators introduced 
in~\eqref{eq:defA} and ~\eqref{eq:defB}.
Set $q_{0}:=\mathrm{A}_{1-\alpha}(\mathrm{B}_{1-\alpha}M)$. Define $P:\B\to\B$ by 
$P v\equiv \int_Y v\, \dd \mu_Y$.
Then,
$$
\hat T(e^{-s})\sim \frac{1}{s^\alpha q_{0}(s)}P\text{ as } s\to 0,
$$
\end{lemma}

\begin{proof} For simplicity we write $\hat R(s), \hat T(s)$ instead of $\hat R(e^{-s})$, 
$\hat T(e^{-s})$. As in  
Subsection~\ref{subsec-spectral} (with $s$ instead of $i\theta$), by (A2) we have
$\|\hat R(s)-\hat R(0)\|\leq Cs^\alpha$, for some $C>0$. This 
together with (A1) implies that there exist $\delta>0$ and a $C^\alpha$ family
of eigenvalues $\lambda(s)$  well defined in $B_\delta(0)$ with $\lambda(0)=1$. Let 
$P(s):\B\to\B$ be the family of 
spectral projections associated with $\lambda(s)$,  with $P(0)=P$.
Let $Q(s)=I-P(s)$ be the family of complementary spectral projections. Since 
$\hat R(s)$ is $C^\alpha$, the same holds for $P(s)$ and $Q(s)$.

We recall the following decomposition of $\hat T(s)$ for $s\in B_\delta(0)$ from 
\cite[Proposition 2.9]{MT12} (extensively used in~\cite{MT12, MT13}):
$$
\hat T(s)=(1-\lambda(s))^{-1}P+(1-\lambda(s))^{-1}(P(s)-P)+(I-\hat R(s))^{-1}Q(s).
$$
By definition, $\|(I-\hat R(s))^{-1}Q(s)\|=O(1)$, as $s\to 0$.
By the argument used in obtaining~\eqref{eq-rediid}(with $s$ instead of $i\theta$)
$$
 1-\lambda(s)=\int_Y(1- e^{-s\tilde\tau})\, \dd m+O(s^{2\alpha}).
$$
It follows from~\cite[Corollary 1]{PK2} (see also~\eqref{eq:Laplace-asy} here) that
$\int_Y(1- e^{-s\tilde\tau})\, \dd m\sim s^{\alpha} q_{0}(s)$. Thus,
\begin{equation*}
(1-\lambda(s))^{-1}\sim \frac{1}{s^\alpha q_{0}(s)}.
\end{equation*}
We already know that the families $P(s)$ and $Q(s)$ are $C^\alpha$. Putting the above 
together, $(I-\hat R(s))^{-1}=s^{-\alpha} q_{0}(s)^{-1}P +E(s)$, where 
$\|E(s)\|=o(s^{-\alpha} q_{0}(s)^{-1})$ and the conclusion follows.
\end{proof}

\begin{pfof}{Proposition~\ref{prop-avop}}Let $v\in \B$. Let $p=A_\alpha^{-1}(1/q_0)$ with 
$A_\alpha$ and $q_0$ given in Lemma~\ref{lemma-hatT} and $z$ a continuity point of $p$. 
It follows from Theorem~\ref{prop-genav} and Lemma~\ref{lemma-hatT} that
\begin{equation} 
\label{eq-unif}
\lim_{n \to \infty}
\frac{\sum_{j=0}^{[ c^{n/\alpha}z]} T_j v}{ c^n}=z^\alpha\, 
p(z)\int_Y v \dd \mu_Y,
\end{equation}
uniformly  on $Y$.
The statement for  H{\"o}lder observables $v: [0,1]\to\R$ supported on any compact set of 
$(0,1]$ follows from~\eqref{eq-unif} together with a word by word repeat of the argument 
used in~\cite[Proof of Theorem 3.6 and first part of Proof of Theorem 
1.1]{MT13}.~\end{pfof}

\section{Appendix}

\subsection{On the discrete form of  (\ref{eq-dgp-distf})}
\label{sec-discr}

Let us assume that the discrete version of (\ref{eq-dgp-distf}) holds, i.e.
\begin{equation*} \label{eq-dgp-distf-int}
\overline F(n) = \frac{\ell(n)}{n^\alpha}  [ M(\delta(n)) + h(n) ],
\end{equation*}
where $\ell : \N \to (0,\infty)$ is a slowly varying sequence, and 
$h: \N \to \R$ is right-continuous error function such that
$\lim_{n \to \infty} h( \lfloor A_{k_n} x \rfloor ) = 0$, whenever $x$ is a continuity 
point of $M$. It is possible to extend the functions $\ell$ and $h$ (still denoted by 
$\ell$ and $h$), such that 
(\ref{eq-dgp-distf}) holds. Indeed, let
\[
\ell(x) = \frac{\ell(\lfloor x \rfloor) x^\alpha}{\lfloor x \rfloor ^\alpha}, \quad
h(x) = h(\lfloor x \rfloor) + M(\delta(\lfloor x \rfloor)) - M(\delta(x)).
\]
Then 
\[
\overline F(x) = \frac{\ell(x)}{x^\alpha}  [ M(\delta(x)) + h(x) ]
= \overline F(\lfloor x \rfloor).
\]
Clearly, $\ell$ is a slowly varying function, so we only have to show that $h$ 
satisfies the conditions after (\ref{eq-dgp-distf}). Let $x$ be a continuity point of 
$M$, and assume that $x \in (1, c^{1/\alpha})$. The general case follows similarly. By 
the definition
\[
h(A_{k_n} x ) = h( \lfloor A_{k_n} x \rfloor) + [M(\delta(\lfloor A_{k_n} x 
\rfloor)) - M(\delta( A_{k_n} x))],
\]
so according to assumption on $h$, it is enough to show that the term in the square 
brackets tends to 0. As $A_{k_{n+1}} / A_{k_n} \to c^{1/\alpha}$, we have for $n$ large 
enough, $\delta ( A_{k_n} x ) = x$, and $\delta ( \lfloor A_{k_n} x \rfloor) =  \lfloor 
A_{k_n} x \rfloor / A_{k_n} \to x$. Since $x$ is a continuity point of $M$, the 
statement follows.

\subsection{Proof of Lemma \ref{lemma:uniform}}
\label{subsec-lemmaunif}

Introduce the notation
\[
\begin{split}
\nu_\lambda( x) & = 1 - \frac{R_\lambda(x)}{R_\lambda(1)} 
= 1- x^{-\alpha} \frac{M(x \lambda^{1/\alpha})}{M(\lambda^{1/\alpha})},
\quad x \geq 1. 
\end{split}
\]
Then $\nu_\lambda$ is a distribution function. Consider the decomposition
\[
G_\lambda(x) = G_{\lambda, 1}(x) * G_{\lambda, 2}(x),
\]
where
\[
G_{\lambda, 1}(x) = e^{-s} \sum_{n=0}^\infty \frac{s^n}{n!} \nu_\lambda^{*n}(x), 
\]
with $s= - R_\lambda(1)$, where $*$ stands for convolution, and $*n$ for $n$th 
convolution power. Simply
\[
\int_0^\infty e^{-u x} \dd G_{\lambda, 1}(x) = 
\exp \left\{ - \int_1^\infty \left( 1 - e^{-u y} \right) \dd R_\lambda(y) \right\} .
\]
Since uniformly in $\lambda$
\[
\overline G_{\lambda, 2} (x) = o(e^{-x}) \quad \text{as } x \to \infty,  
\]
we have that uniformly in $\lambda \in [c^{-1},1]$
\[
\overline G_{\lambda} (x) \sim \overline  G_{\lambda, 1} (x). 
\]
Therefore, to prove the statement we have to show that
\[
\overline  G_{\lambda, 1} (x) \sim  \frac{M ( x \lambda^{1/\alpha} )}{x^\alpha}
\]
holds uniformly in $\lambda \in [c^{-1},1]$.
From the proof of implication (ii) $\Rightarrow$ (iii) of Theorem 3 in \cite{EGV} we see 
that it is enough to show that subexponential property and the Kesten bounds hold 
uniformly in $\lambda \in [c^{-1},1]$, i.e., with 
$\overline \nu_\lambda (x) = 1 - \nu_\lambda(x)$
\begin{equation} \label{eq-subexp-n-uniform}
\lim_{x \to \infty} \sup_{\lambda \in [c^{-1},1]} \left| 
\frac{\overline {\nu^{*n}_\lambda}(x)}{\overline \nu_\lambda(x)} - n \right| =0,
\end{equation}
and for any $\varepsilon > 0$, there exists $K$, such that for all $n \in \N$ and 
$\lambda \in [c^{-1},1]$
\begin{equation} \label{eq-Kesten-uniform}
\overline {\nu^{*n}_{\lambda}} (x) \leq K ( 1 + \varepsilon)^n \overline \nu_\lambda(x).
\end{equation}
According to Theorem 3.35 and Theorem 3.39 (with $\tau \equiv n$) by Foss, Korshunov, and 
Zachary \cite{FKZ} both (\ref{eq-subexp-n-uniform}) and (\ref{eq-Kesten-uniform}) hold if
\begin{equation} \label{eq-subexp-uniform}
\lim_{x \to \infty} \sup_{\lambda \in [c^{-1},1]} \left|
\frac{\overline{ \nu_\lambda * \nu_\lambda} (x)}{\overline \nu_\lambda(x)} - 2 
\right| = 0. 
\end{equation}

Now we prove (\ref{eq-subexp-uniform}). Write
\begin{equation} \label{eq-nu2-decomp}
\frac{\overline{\nu_\lambda * \nu_\lambda}(x)}{\overline \nu_\lambda(x)}  =  
\int_1^{x-1} \frac{\overline \nu_\lambda(x-y)}{\overline \nu_\lambda(x)} \dd 
\nu_\lambda(y) + \frac{\overline \nu_\lambda(x-1)}{\overline \nu_\lambda(x)}.
\end{equation}
By the logarithmic periodicity of $M$ the second term can be written as
\[
\frac{\overline \nu_\lambda(x-1)}{\overline \nu_\lambda(x)} 
= \left( \frac{x}{x-1} \right)^\alpha 
\frac{M \left( c^{\alpha^{-1} \{ \log_c x^\alpha \}} \lambda^{1/\alpha} ( 1 - x^{-1}) 
\right)}
{M \left( c^{\alpha^{-1} \{ \log_c x^\alpha \}} \lambda^{1/\alpha} \right)},
\]
which goes to 1 uniformly in $\lambda$ due to the continuity of $M$ (a continuous 
function is uniformly continuous on compacts).
In order to handle the first term in (\ref{eq-nu2-decomp}) choose $\delta > 0$ 
arbitrarily small, and $K$ so large that 
\[
\sup_{\lambda \in [c^{-1},1]} \overline \nu_\lambda(K) < \delta.
\]
As before
\[
\int_1^K \left| \frac{\overline \nu_\lambda(x-y)}{\overline \nu_\lambda(x)} - 1
\right| \dd \nu_\lambda(y) \to 0,
\]
uniformly in $\lambda \in [c^{-1},1]$. We show that the integral on $[K, \infty)$ is 
small. 
Putting $C = \sup_{y} M(y) / \inf_{y} M(y)$ and integrating by parts
\begin{equation} \label{eq-uniform-aux1}
\begin{split}
& \int_K^{x- 1} \frac{\overline \nu_\lambda(x-y) }{\overline \nu_\lambda(x)} \dd 
\nu_\lambda(y) 
 \leq C \int_K^{x-1} \left( \frac{x}{x-y} \right)^{\alpha} \dd \nu_\lambda(y) \\
& = C \left[ \overline \nu_\lambda(K) \left( \frac{x}{x-K} \right)^{\alpha}
- \overline \nu_\lambda(x-1) x^\alpha + 
x^{\alpha} \int_K^{x-1} \overline \nu_\lambda(y) \alpha (x-y)^{-\alpha -1} \dd y \right].
\end{split}
\end{equation}
The first term in the bracket is small. The uniform continuity of $M$ on 
compact sets, and its strict positivity implies that there is $\delta'>0$ small 
enough, such that for all $y \in [c^{-1/\alpha}, c^{1/\alpha}]$
\[
\begin{split}
(1- \delta) M(y) & \leq \inf_{0 \leq u \leq \delta'} M( (1-u) y ) \\
& \leq  \sup_{0 \leq u \leq \delta'} M( (1-u) y ) \leq (1 + \delta) M(y). 
\end{split}
\]
Thus, using also that $\int_{\delta'}^1 u^{-\alpha-1} (1-u)^{-\alpha} \dd u < \infty$, we 
obtain
\[
\begin{split}
& x^{\alpha} \int_K^{x-1} \overline \nu_\lambda(y) \alpha (x-y)^{-\alpha -1} \dd y \\
& = \frac{\alpha}{x^{\alpha} M(\lambda^{1/\alpha})} \int_{1/x}^{1- K/x} 
\frac{M\left( (1-u) c^{\alpha^{-1} \{\log_c x^\alpha \} }\lambda^{1/\alpha} \right) }
{u^{1+\alpha} (1-u)^{\alpha}} \dd u \\
& = \frac{\alpha}{x^{\alpha} M(\lambda^{1/\alpha})} \int_{x^{-1}}^{\delta'} 
\frac{M \left( (1-u) c^{\alpha^{-1} \{\log_c x^\alpha \} }  
\lambda^{1/\alpha} \right) }{u^{\alpha+1} (1-u)^{\alpha}} \dd u + O(x^{-\alpha}) \\
& \leq \frac{1 + \delta}{(1-\delta')^\alpha} 
\frac{M(x \lambda^{1/\alpha})}{M(\lambda^{1/\alpha})} + O(x^{-\alpha}).
\end{split}
\]
The lower bound follows similarly. Substituting back into (\ref{eq-uniform-aux1})
\[
\limsup_{x \to \infty} \sup_{\lambda \in [c^{-1},1]}
\int_K^{x- 1} \frac{\overline \nu_\lambda(x-y) }{\overline \nu_\lambda(x)} \dd 
\nu_\lambda(y) \leq C \left[ \delta +  
\max \left\{ \frac{1+\delta}{(1- \delta')^\alpha} -1 , \delta \right\} \right].
\]
Since $\delta > 0$ and $\delta'>0$ are arbitrarily small, the statement follows.

\subsection{A technical result used in the proof of Proposition~\ref{prop-genav}}
\label{subsec-lemma1}

\begin{lemma} \label{lemma-poly}
Put $g = 1_{[e^{-1},1]}$. For any $\delta > 0$ and $\varepsilon > 0$ 
there exist polynomials $Q_1$ and $Q_2$ such that
\[
Q_1(x) \leq g(x) \leq Q_2(x), \quad x \in [0,1],
\]
and for any measure
$\mu$ on $(0,\infty)$ such that $\int_0^\infty e^{-x} \mu(\dd x) < \infty$,
\[
\begin{split}
& \int_0^\infty \left[ Q_2(e^{-x}) - g(e^{-x} ) \right] \mu(\dd x) < 
\varepsilon \int_0^\infty e^{-x} \mu(\dd x) + 
\mu((1-\delta, 1+\delta)) \\
& \int_0^\infty \left[ g(e^{-x}) - Q_1(e^{-x}) \right] \mu (\dd x) < 
\varepsilon \int_0^\infty e^{-x} \mu(\dd x) +  
\mu((1-\delta, 1+\delta)).
\end{split}
\]
\end{lemma}

\begin{proof}
Fix $\varepsilon > 0$ and $\delta > 0$. Let
\[
g_2 (x) =
\begin{cases}
0, & x \leq e^{-1} - \delta', \\
\frac{e}{\delta'} (x - e^{-1} + \delta'), & x \in [e^{-1} - \delta', e^{-1}], \\
x^{-1}, & x \in [e^{-1}, 1],
\end{cases}
\]
where $\delta' > 0$ is chosen such that $- \log (e^{-1} - \delta') < 1 + \delta$.
Then $g_2$ is a continuous function on $[0,1]$, and $x g_2(x) \geq g(x)$. 
By the approximation theorem of Weierstrass, there is a polynomial $r_2(x)$, such that
\[
\sup_{x \in [0,1]} | r_2(x) - (g_2(x) + \varepsilon/2) | \leq \frac{\varepsilon}{2}. 
\]
Let $Q_2(x) = x r_2(x)$. By the choice of $r_2$
\[
 0 \leq Q_2(x) - x g_2(x) \leq \varepsilon x.
\]
Moreover,
\[
\begin{split}
Q_2(x) - g(x) & =  Q_2(x) - x g_2(x) + x g_2(x) - g(x) \\
& \leq \varepsilon x + 1_{[e^{-1} - \delta', e^{-1}]}(x).
\end{split}
\]
Therefore
\[
\begin{split}
\int_0^\infty \left[ Q_2(e^{-x}) - g(e^{-x}) \right] \mu(\dd x) 
& \leq \varepsilon \int_0^\infty e^{-x} \mu(\dd x) + 
\mu \left( [1, - \log (e^{-1} - \delta')] \right) \\
& \leq \varepsilon \int_0^\infty e^{-x} \mu(\dd x) + \mu ( [1, 1+ \delta)).
\end{split}
\]

The construction of $Q_1$ is similar. Choose
\[
g_1(x) =
\begin{cases}
0, & x \leq e^{-1}, \\
(\delta')^{-1} (e^{-1} + \delta')^{-1} (x-e^{-1}), & x \in [e^{-1}, e^{-1}+\delta'], \\
x^{-1}, & x \geq e^{-1} + \delta',
\end{cases}
\]
and let $r_1$ be a polynomial such that
\[
\sup_{x \in [0,1]} | r_1(x)  - (g_1(x) -\varepsilon/2)| \leq \frac{\varepsilon}{2}.
\]
The same proof shows that $Q_1(x) = x r_1(x)$ satisfies the stated properties.
\end{proof}


\subsection{Verifying that~\eqref{eq:alphan} holds}
\label{subsec-lemma2}

To ease the notation put
\[
a = \frac{2\pi\alpha}{\log c}. 
\]
Then, recall
\[
\xi_n = \xi(n) = \frac{1}{2n^{\alpha}} \left( 1+ 2\eps \sin (a \log n) \right).
\]
The first derivative is
$$
\xi'(n) = -\frac{1}{2n^{\alpha+1}}
\left( \alpha (1+ 2\varepsilon \sin (a \log n))- 2 \varepsilon a \cos (a \log n) \right)  
< 0,
$$
whenever $\varepsilon$ is small enough. Long but straightforward calculation gives
\begin{equation} \label{eq:Dxi-Taylor}
\Delta \xi_n = \xi_{n} - \xi_{n-1} =
\frac{n^{-\alpha -1}}{2} \left[ x_0(n) + \frac{x_1(n)}{n} + \frac{x_2(n)}{n^2} + O(n^{-3})
\right],
\end{equation}
where
\[
x_0(n) = \alpha ( 1 + 2 \varepsilon \sin ( a \log n) )  - 2 \varepsilon a \cos (a \log n),
\]
and
\[
x_i(n) = c^i_0 + c^i_1 \sin (a \log n) + c^i_2 \cos (a \log n), \quad i = 1,2,
\]
where $c^i_j$ are constants, whose actual value is not important for us. Note that 
$x_0(n)$ comes from the first derivative, and we use it frequently that 
\begin{equation} \label{eq:x0>0}
x_0(n) \geq \alpha - 2 \varepsilon (a + \alpha ) > 0
\end{equation}
for $\varepsilon> 0$ small enough. From \eqref{eq:Dxi-Taylor} we deduce that
\[
\frac{\Delta \xi_{n-2}}{\Delta \xi_{n-1}} = 1 + \frac{\alpha_n}{n} + \frac{r_n}{n^2}
+ R_n, 
\]
with $R_n = O(n^{-3})$, and
\begin{equation*}
\alpha_n = H_1( a \log n), \quad 
r_n = H_2( a \log n), 
\end{equation*}
where
\begin{equation} \label{eq:Hs}
\begin{split}
H_1(x) & = 1 + \alpha - 2 \varepsilon a 
\frac{a \sin x + \alpha \cos x}
{\alpha ( 1 + 2 \varepsilon \sin x) - 2 \varepsilon a \cos x} \\
H_2(x) & = \frac{a_0^{2} + a_1^{2} \sin x + a_2^2 \sin (2x) + b_1^{2} \cos x
+  b_2^{2} \cos (2x)}
{(\alpha ( 1 + 2 \varepsilon \sin x) - 2 \varepsilon a \cos x)^2}
\end{split}
\end{equation}
with some constants $a_j^{2}, b_j^{2}$, whose value is not important. By \eqref{eq:x0>0} 
the denominators in $H_1, H_2$ are strictly positive, therefore $H_1$ and $H_2$ are 
continuous smooth ($C^{\infty}$) functions. This implies that
$\alpha_n = 1 + \alpha + O(\varepsilon)$, $r_n=O(1)$,
\[
\alpha_n - \alpha_{n-1} = O(n^{-1}), \ 
\alpha_{n-1} + \alpha_{n+1} - 2 \alpha_n= O(n^{-2}), \
r_n - r_{n-1} = O(n^{-1}). 
\]
This is everything we need for the construction of $f_\eps$ in Subsection 
\ref{subsec-defdyn}.

\subsection{Distortion properties for $F$}

Let $J_n := [ (\xi_n+1)/2, (\xi_{n-1}+1)/2 )$ be the intervals on which $F := F_\eps$ is 
continuous. 

\begin{lemma} \label{lemma-distfn}
There exists $K>0$ such that $\frac{F''(x)}{F'(x)^2} \leq K$ for all $n$ and all $x \in 
J_n$. In particular, $F|_{J_n}$ can be extended to $\overline{J_n}$ for each $n$ so that 
$\frac{F''(x)}{F'(x)^2} \leq K$ for all $x \in \overline J_n$.
\end{lemma}
\begin{proof}Given the map $f_{\eps, n}$ in Subsection~\ref{subsec-defdyn}, it is easy to 
see that for $x \in [\xi_{n-1}, \xi_n]$, $n\ge 1$,
$$
f_{\eps,n}''(x) = \frac{A_n}{\xi_{n-1}-\xi_n} = O(n^{\alpha-1})
\quad \text{ and } \quad |f_{\eps, n}'(x)-(1+\frac{\alpha_n}{n})| = O(n^{-2}),
$$
where the derivatives at the end-points are interpreted as one-sided derivatives.
From \eqref{eq:Hs} at the end of the previous subsection, we know that 
$\alpha_n = 1+\alpha+  O(\eps)$
as $n \to \infty$ and $\eps \to 0$.
Since $\alpha > 0$, we can choose $\delta > 0$ small enough such that
$(1 + \alpha ) (1 - \delta ) > 1$. For $n$ large enough
$f_{\eps,n}'(x) \geq 1+ \frac1n(1+\alpha)(1-\delta) > 1$.
It follows that
\begin{equation*}
D(f_{\eps,n}) := \frac{f_{\eps,n}''}{(f_{\eps,n}')^2} \text{ is uniformly bounded.}
\end{equation*}
Next, compute that for any two $C^2$ functions $g,h$,
$$
D(g \circ h) = D(g) \circ h + \frac{1}{g'\circ h} D(h).
$$
Applying this to $g = f^{n-1}$ and $h = f$, gives
$$
D(f^n) = D(f^{n-1}) \circ f + \frac{1}{(f^{n-1})' \circ f} D(f).
$$
Write $x_k = f_\eps^k(x)$ for $k \geq 0$. For some $C = C(\delta) > 0$
\begin{eqnarray*}
(f_\eps^{n-1})'(x) 
& = & f_\eps'(x_{n-2}) f_\eps'(x_{n-3}) \cdots f_\eps'(x_0)  \\
& \geq & C \left( 1+\frac{(1+\alpha)(1- \delta )}{n-1} \right)  
\left( 1+\frac{(1+\alpha)(1- \delta )}{n-2} \right) 
\cdots 2 \\
&=& 2 C \exp\left(\sum_{k=2}^n \log  \left[ 1+\frac{(1+\alpha)(1- \delta )}{k-1} 
\right] \right) 
\\
&\sim & 2 C \exp(((1+\alpha)(1- \delta ))\log n + C_n) \\
& \geq & C' n^{(1+\alpha)(1- \delta )},
\end{eqnarray*}
where $(C_n)$ is a bounded sequence and $C' >0$. We get 
$$
D(f^n) \leq  D(f^{n-1}) \circ f + \frac{1}{C' n^{(1+\alpha)(1- \delta )} } D(f).
$$
By induction, 
$$
D(F|_{J_n}) \leq  D(f^n|_{J_n}) \ll D(f) \sum_{k=2}^{n-1} 
\frac{1}{C' k^{(1+\alpha)(1- \delta )}},
$$
which is bounded in $n$ since the exponent $(1+\alpha)(1- \delta ) > 1$.
\end{proof}

\bigskip
\noindent \textbf{Acknowledgements.}
PK's research was supported by the J\'anos Bolyai Research Scholarship of the Hungarian 
Academy of Sciences, and by the NKFIH grant FK124141. DT would like to thank CNRS for enabling her a three month visit to
IMJ-PRG, Pierre et Marie Curie University, where  her research on this project began.

\bibliographystyle{abbrv}
\bibliography{DKlike}

\end{document}